\documentclass[11pt,a4paper]{article}

\title{Rigorous and heuristic treatment of sensitive singular perturbations arising in elliptic shells}
\author{
{Yuri V. Egorov} \footnote{Laboratoire MIP, Universit\'{e} Paul Sabatier, 31062 Toulouse Cedex 9, France.(egorov@mip.ups-tlse.fr)},
 {Nicolas Meunier} \footnote{Laboratoire MAP5,
Universit\'{e} Paris Descartes (Paris V), 45 Rue des Saints
P\`{e}res,75006 Paris, France. (nicolas.meunier@parisdescartes.fr)} { and
Evariste Sanchez-Palencia} \footnote{Laboratoire de Mod\'{e}lisation en M\'{e}canique, Universit\'{e} Pierre et
Marie Curie (Paris VI),
 4 place Jussieu,
75252 Paris,
FRANCE, (sanchez@lmm.jussieu.fr)}}

\usepackage[T1]{fontenc}

\usepackage[english,francais]{babel}
\usepackage{amsmath}
\usepackage{amssymb}
\usepackage{amsthm}

\newtheorem{theorem}{Theorem}[section]
\newtheorem{lemma}[theorem]{Lemma}
\newtheorem{proposition}[theorem]{Proposition}

\theoremstyle{definition}
\newtheorem{definition}[theorem]{Definition}
\theoremstyle{remark}
\newtheorem{remark}{Remark}

\numberwithin{equation}{section}

\newcommand{\D}{\ {\rm d}}

\newcommand{\Norm}[1]{\lVert #1 \rVert}
\newcommand{\scalprod}[2]{( #1 , #2 )}
\newcommand{\dualprod}[2]{ \langle #1 , #2 \rangle}
\newcommand{\R}{\mathbb{R}}
\newcommand{\C}{\mathbb{C}}
\newcommand{\Z}{\mathbb{Z}}

\begin{document}
\maketitle
\abstract{We consider singular perturbations of elliptic systems depending on a
parameter $\varepsilon$ such that, for $\varepsilon=0$ the
boundary conditions are not adapted to the equation (they do not
satisfy the Shapiro - Lopatinskii condition). The limit holds only in very abstract spaces out of distribution theory involving
complexification and non-local phenomena. This system appears in the thin shell theory when the middle surface is elliptic and the shell is fixed on a part of the boundary and free on the rest. We use a heuristic reasoning applying some simplifications which allow to reduce the original problem in a domain to another problem on its boundary. The novelty of this work is that we consider systems of partial differential equations while in our previous work we were dealing with single equations.
}

\section{Introduction}
This paper is devoted to a very singular kind of perturbation problems arising in thin shell theory. Up to our knowledge, it is disjoint of relevant and well known contributions of V. Mazya on perturbation of domains and multistructures for elliptic problems including the Navier - Stokes system (\cite{MazyaNazarov}, \cite{KozlovMazya}, \cite{MazyaSlutskii}), as the pathological feature of our problem is concerned with ill-posedness of the limit problem, generating singularities out of the distribution space. So, it may be considered as a contribution to enlarge perturbation theory of Mazya.  More precisely, the main purpose of this paper is to generalize the previous work done on equations, see \cite{EgorovMeunierSanchez}, \cite{MeunierSanchez} to systems of partial differential equations. The motivation for studying that kind of problems comes from the shell theory. It appears that when the middle surface is elliptic (both principal curvatures have same sign) and is fixed on a part $\Gamma _0$ of the boundary and free on the rest $\Gamma _1$, the "limit problem" as the thickness $\varepsilon $ tends to zero is elliptic, with boundary conditions satisfying Shapiro - Lopatinskii (SL hereafter) on $\Gamma _0$ but not satisfying it on $\Gamma_1$. In other words, the "limit problem" for $\varepsilon=0$ is
highly ill-posed. 
This pathological behavior arises only as $\varepsilon=0$. In fact, for $\varepsilon>0$ the
problem is "classical". 

In such kind of situations, the limit problem has no solution
within classical theory of partial differential equations, which
 uses distribution theory. It is sometimes possible to prove the
convergence of the solutions $u^{\varepsilon}$ towards some limit
$u^{0}$, but this "limit solution" and the topology of the
convergence are concerned with abstract spaces not included in the
distribution space.

The variational problem we are interested in is:
\begin{equation}\label{intro1}
\left\{\begin{array}{l}\textrm{Find } u^\varepsilon  \in V
\textrm{ such that, } \forall v \in V\\
a(u^\varepsilon ,v)+\varepsilon ^2 b(u^\varepsilon ,v)=\dualprod{f}{v},
\end{array}\right.
\end{equation}
or, equivalently, the minimization in $V$ of the functional 
$$a(u,u)+\varepsilon ^2 b(u,u)-2\dualprod{f}{u},$$
where $f \in V'$ is given and  the brackets denote the duality between $V'$ and $V$.

This is the Koiter model of shells, $\varepsilon $ denoting the relative thickness. The corresponding energy space $V$ is a classical Sobolev space. 

 The limit boundary partial differential system associated with (\ref{intro1}) when $\varepsilon =0$  is elliptic and ill-posed.

Let us consider formally the variationnal problem of the membrane problem (i.e. $\varepsilon =0$):
\begin{equation}\label{sensitiveconvergenceajout1}
\left\{\begin{array}{l}\textrm{Find } u  \in V_a
\textrm{ such that, } \forall v \in V_a\\
a(u ,v)=\dualprod{f}{v},
\end{array}\right.
\end{equation}
where
$V_a$ is the abstract completion of the "Koiter space" $V$  with the norm $ \| v\| _{a}=a(v,v)^{1/2}$, it is to be noted that   the elements of $V_a$ are not necessarly distributions. The term "sensitive" originates from the fact that this latter problem is unstable. Very small  and smooth variations of $f$ (even in $\mathcal{D}(\Omega)$) induce modifications of the solution which are large and singular (out of the distribution space).

The plan of the article is as follows. After recalling the Koiter shell model (Section \ref{Koiter}), we recall the definitions of ellipticity and the Schapiro-Lopatinskii condition for systems elliptic in the Douglis-Nirenberg sense (Section \ref{sectionSL}).  In Section \ref{4pb}, we study four systems of partial differential equations which are involved in our study of shell theory. These systems are the rigidity system, the membrane tension system, the membrane system and the Koiter shell system.

In section \ref{Sectionmodel}, we study a sensitive perturbation problem arising in Koiter linear shell theory and we briefly recall some abstract convergence results. In Section \ref{heuristics}, we report the heuristic procedure of \cite{EgorovMeunierSanchez}. In this latter article,  we addressed a model problem including
a variational structure, somewhat analogous to the shell problem studied here,
but simpler, as concerning an equation instead of a system. It is
shown that the limit problem involves in particular an elliptic
Cauchy problem. This problem was handled in both a rigorous (very
abstract) framework and using a heuristic procedure for exhibiting
the structure of the solutions with very small $\varepsilon$. The
reasons why the solution goes out of the distibution space as
$\varepsilon$ goes to $0$ are then evident. The heuristic procedure is very much analogous to the method of construction of a parametrix in elliptic problems \cite{Taylor}, \cite{Egorov}:

-Only principal (with higher differentiation order) terms are
taken into account.

-Locally, the coefficients are considered to be constant, their
values being frozen at the corresponding points.

-After Fourier transform ($x \to \xi$), terms with small
$\xi$ are neglected with respect to those with larger $\xi$ (which
amounts to taking into account singular parts of the solutions
while neglecting smoother ones). We note that this approximation,
aside with the two previous ones, lead to some kind of "local
Fourier transform" which we shall use freely in the sequel.

Another important feature of the heuristics is a previous
drastic restriction of the space where the variational problem is
handled. In order to search for the minimum of energy, we only
take into account functions such that the energy of the limit
problem is very small. This is done using a boundary layer method
within the previous approximations, i.e. for large $|\xi|$. This
leads to an approximate simpler formulation of the problem for
small $\varepsilon$, where  it is apparent that the
limit problem involves a smoothing operator and cannot have a
solution within distribution theory.


Notations are standard. We denote
\begin{equation}\label{m3}
\partial_{k}=\frac{\partial}{\partial x_{k}},\ \ k=1,2,
\end{equation}
and
\begin{equation}\label{m3bis}
D_{k}=-i\frac{\partial}{\partial x_{k}},\ \ k=1,2\textrm{ and }
D^\alpha =D_1^{\alpha _1}D_2^{\alpha _2}, \ \alpha=(\alpha _1,
\alpha _2 )\in \Z _+^2.
\end{equation}
Moreover, the definition of the Sobolev space $H^s(\Gamma)$, $s\in
\R$, where $\Gamma $ is a one dimensional compact manifold is
classical  using a partition of unity and local mappings.

The inner product and the duality products associated with a space $V$ and its dual $V'$
will be denoted by 
$\scalprod{.}{.}$ 
and $\dualprod{.}{.}$ respectively. 

The usual convention of summation of repeated indices is used. Greek and latin indices will belong to the sets $\{1,2\}$ and  $\{1,2,3\}$ respectively.


\section{Generalities on the Koiter shell model}\label{Koiter}
\noindent

Let $\Omega $ be a bounded open set of $\R^2$ with smooth boundary $\Gamma $. Let $\mathcal{E}^3$ be the euclidean space referred to the othonormal frame $(O,e_1, e_2, e_3)$. We consider the shell theory in the framework of the Koiter theory and more precisely the mathematical framework of this linear theory. The middle surface $S$ of the shell is the image in $\mathcal{E}^3$ of $\Omega $ for the map 
$$\varphi \ : \ (y^1,y^2) \in \overline{\Omega} \to \varphi (y) \in \mathcal{E}^3.$$

The two tangent vectors of $S$ at any point $y$  are given by:
$$a_\alpha = \partial _\alpha \varphi , \ \alpha \in \{1,2\},$$
where $\partial _\alpha $ denotes the differentiation with respect to $y^\alpha $,  while the unit normal vector is: 
$$a_3= \frac{a_1 \wedge a_2}{\Norm{a_1 \wedge a_2}}.$$
For simplicity, we omitted $y$ in the previous notation ($a_\alpha (y)$).

The middle surface $S$ is assumed to be smooth ($\mathcal{C}^\infty$) and we may consider in a neighbourhood of it a system of "normal coordinates" $y^1, y^2, y^3$, when $y^3$ is the normal distance to $S$. More precisely we consider a shell of constant thickness $\varepsilon  $, i.e. it is the set 
$$C=\{M \in \mathcal{E}^3, \ M=\varphi (y^1,y^2)+y^3a_3, \ (y^1,y^2)\in \Omega, -\frac{1}{2} \varepsilon <y^3<\frac{1}{2} \varepsilon \}.$$
Under these conditions, let $u=u(y^1,y^2)$ be the displacement vector of the middle surface of the shell. In the linear theory of shells, which is our framework here, the displacement vector is assumed to describe the first order term of the mathematical expression as the thickness $\varepsilon $ is small, see \cite{Ciarlet, San1}. 
\begin{remark}
In the sequel smooth should be understood in the sense of $\mathcal{C}^\infty$.
\end{remark}
\begin{remark}
We consider here the case where the surface is defined by only one chart but this could be easily generalized to the case of several charts (atlas). 
\end{remark}

More precisely, since we consider the case where  $u$ is supposed to be small, the Koiter theory is described in terms of the \textit{deformation tensor}  (or strain tensor) $\gamma _{\alpha \beta}$ of the middle surface:
$$\gamma _{\alpha \beta}=\frac{1}{2} (\tilde{a}_{\alpha \beta} -a_{\alpha \beta})$$
and the \textit{change of curvature tensor} $\rho _{\alpha \beta}$:
$$\rho _{\alpha \beta}=\tilde{b} _{\alpha \beta}-b _{\alpha \beta}.$$
In the previous definitions, the expressions $a _{\alpha \beta}$ (resp. $\tilde{a} _{\alpha \beta}$) denote the coefficients of the first fundamental form of the middle surface before (resp. after) deformation:
$$a _{\alpha \beta}=a _{\alpha } \cdot a _{ \beta}= \partial _{\alpha } \varphi \cdot \partial _{ \beta} \varphi,$$
and $b _{\alpha \beta}$ (resp. $\tilde{b} _{\alpha \beta}$) the coefficients of the second fundamental form  accounting for the curvatures before (resp. after) deformation:
$$b _{\alpha \beta}=-a _{\alpha } \cdot \partial _{ \beta}  a _{ 3}= a_3 \cdot \partial _{\beta } a_\alpha=a_3 \cdot \partial _{\alpha } a_\beta =b _{\beta \alpha},$$
due to the fact that $a _{\alpha } \cdot  a _{ 3}=0$.

The dual basis $a^i$ is defined by 
$$a_i \cdot a^j=\delta ^j_i,$$
where $\delta $ denotes the Kronecker symbol. The contravariant components $a^{i j}$ of the metric tensor are:
$$a^{ij}= a^i \cdot a^j,$$
and  $a_{i j}$  are used to write covariant components of vectors and tensors in the usual way. 
Finally, the tensors $\gamma $ and $\rho $ take the form:
\begin{eqnarray}\label{gamma}
\gamma _{\beta \alpha } (u)&=&\gamma _{\alpha \beta } (u)= \frac{1}{2} ( u _{\alpha| \beta}+ u _{\beta| \alpha})-b_{\alpha \beta} u_3,\\
\rho _{\alpha \beta } (u)&=& u _{3| \alpha \beta}+b^\lambda _{\beta  | \alpha}u _\lambda +b^\lambda _{\beta} u_{\lambda | \alpha}+b^\lambda _\alpha u _{\lambda | \beta} -b^\lambda _\alpha b_{\lambda \beta} u_3,\label{gammabis}
\end{eqnarray}
where $\partial _\alpha a_3=b_\alpha ^\gamma a_\gamma $, $b_\alpha ^\beta = a^{\beta \sigma } b_{\alpha \sigma}$, $_{.| \alpha}$ denotes the \textit{covariant differentiation} which is defined by 
\begin{equation}
\left\{\begin{array}{l}
u _{\alpha| \beta}=\partial _{ \beta }u_\alpha - \Gamma ^\lambda _{\alpha \beta }u_\lambda \\
u _{3| \beta }=\partial _{ \beta }u_3,
\end{array}\right.
\end{equation}   
and 
\begin{equation}
\left\{\begin{array}{l}
b^{ \lambda }_{\alpha| \beta}=\partial _{ \alpha }b_\beta ^{\lambda } +\Gamma ^\lambda _{\alpha \nu }b_{\beta }^{\nu} - \Gamma ^\nu _{\beta \alpha }b^{\lambda }_{\nu} =b^{ \lambda }_{\beta| \alpha}\\
u _{3| \alpha \beta}=\partial _{\alpha \beta }u_3 -\Gamma ^\lambda _{\alpha \beta }\partial _\lambda u_3 ,
\end{array}\right.
\end{equation}   
where $ \Gamma ^\alpha _{\beta \gamma }$ are the Christoffel symbols of the surface
$$   \Gamma ^\alpha _{\beta \gamma } =\Gamma ^\alpha _{\gamma \beta }= a^\alpha \cdot \partial _\beta  a_\gamma =a^\alpha \cdot \partial _\gamma  a_\beta .$$
   
Let us now define the energy of the shell in the Koiter framework. It consists of two  bilinear forms $a$ and $b$: $a$  corresponds to a \textit{membrane strain energy} and $b$ is a \textit{bending energy} (which acts as a perturbation term). More precisely, $a$ is defined by
\begin{equation}\label{defedeaajout}
a(u,v)=\int_S A^{\alpha \beta \lambda \mu}\gamma _{\lambda \mu } (u)\gamma _{\alpha \beta } (\overline{v}) \D s,
\end{equation}  
where $A^{\alpha \beta \lambda \mu}$ are the \textit{membrane rigidity} coefficients which we assume to be smooth on $\Omega$. Moreover, we assume that some symmetry holds
\begin{equation}\label{sym}
A^{\alpha \beta \lambda \mu}=A^{ \lambda \mu \alpha \beta}=A^{ \mu \lambda \alpha \beta  }.
\end{equation}

Defining the \textit{membrane stress tensors} by
\begin{equation}\label{defdeA}
T^{\alpha \beta }(u)= A^{\alpha \beta \lambda \mu} \gamma _{\lambda \mu } (u),
\end{equation}
using the symmetry of $\gamma$, we immediately see that 
\begin{equation}\label{symdeA1}
T^{\alpha \beta }(u)=T^{\beta \alpha }(u),
\end{equation}
and
\begin{equation}\label{symdeA2}
a(u,v)=\int_S T^{\alpha \beta }(u)\gamma _{\alpha \beta } (\overline{v}) \D s=\int_S \gamma_{\alpha \beta }(u)T ^{\alpha \beta } (\overline{v}) \D s.
\end{equation}

Furthermore, we assume that a coercivity condition holds uniformly on the surface:
\begin{equation}\label{coercivity}
A^{\alpha \beta \lambda \mu} \xi _{\alpha \beta } \xi _{\lambda \mu} \ge C \Norm{\xi}^2, \ C>0.
\end{equation}

\begin{remark}
It is to be noticed that there are two different symmetries on $A$: the first one $A^{\alpha \beta \lambda \mu}=A^{ \lambda \mu \alpha \beta}$ is necessary to exchange $u$ and $v$ in (\ref{symdeA2}) while the second $A^{ \lambda \mu \alpha \beta}=A^{ \mu \lambda \alpha \beta  }$ is used to obtain (\ref{symdeA1}) but is not necessary in order to obtain (\ref{symdeA2}) since we could use the symmetry of $ \gamma$.
\end{remark}

Analogously, we define the bilinear form $b$ which corresponds to the bending energy of the shell and which will act as a perturbation term:   
\begin{equation}
b(u,v)=\int_S B^{\alpha \beta \lambda \mu}\rho _{\lambda \mu } (u)\rho _{\alpha \beta } (\overline{v}) \D s,
\end{equation}  
where $B^{\alpha \beta \lambda \mu}$ are the \textit{bending rigidity} coefficients which we assume to be smooth on $\Omega$ and to have the same properties (\ref{sym}) and (\ref{coercivity}) as $A$, namely
\begin{equation}\label{symbending}
B^{\alpha \beta \lambda \mu}=B^{ \lambda \mu \alpha \beta}=B^{ \mu \lambda \alpha \beta  },
\end{equation}
and
\begin{equation}\label{coercivitybending}
B^{\alpha \beta \lambda \mu} \xi _{\alpha \beta } \xi _{\lambda \mu} \ge C \Norm{\xi}^2
\end{equation}
uniformly on the surface.   

Similarly to $a$ we can write  
\begin{equation}\label{defdeB1}
b(u,v)=\int_S M^{\alpha \beta }(u)\rho _{\alpha \beta } (\overline{v}) \D s,
\end{equation}
where the \textit{bending stress tensors} are
\begin{equation}\label{defdeB}
M^{\alpha \beta }(u)=B^{\alpha \beta \lambda \mu} \rho _{\lambda \mu } (u).
\end{equation}

In this work, we will restrict ourselves to the case of elliptic surface, i.e. we will always assume that the coefficients $b_{\alpha \beta }$ are such that 
\begin{equation}\label{elliptic}
b_{11}b_{22}-b_{12}^2 >0 \textrm{ uniformly on } S \textrm{ and } b_{11}>0.
\end{equation}

Let us finish this introduction by topoligical considerations,  the boundary $\partial \Omega = \Gamma _0
\cup \Gamma _1$ is assumed to be smooth (i.e. of class $\mathcal{C}^\infty$) in the variable $y=(y^1,y^2)$, where $\Gamma _0 $
and $\Gamma _1$ are disjoint; they are one-dimensional compact
smooth manifolds without boundary, then diffeomorphic to the unit circle.

We consider the following variational problem (which has possibly only a formal sense)
\begin{equation}\label{pbvar}
\left\{\begin{array}{l}\textrm{Find } u ^{\varepsilon }\in V \textrm{ such that, }
\forall v \in V\\
a(u^{\varepsilon } ,v)+\varepsilon ^2 b(u^{\varepsilon }, v) =\dualprod{f}{v},
\end{array}\right.
\end{equation}
with $a$ and $b$ defined by (\ref{defedeaajout}) and (\ref{defdeB1})
where the space $V$ is
the "energy space" with the essential boundary conditions on $\Gamma _0$
\begin{equation}\label{defdeV}
V=\{v; \ v_\alpha  \in H^1(\Omega ), \ v_3 \in H^2(\Omega); \ v_{| \Gamma _0}
=0 \textrm{ in the sense of trace}\}.
\end{equation}
\begin{remark}
The essential boundary conditions on $\Gamma _0$
(\ref{defdeV}) corresponds to the case of the fixed boundary of the shell. Other boundary conditions could have been considered such as:
\begin{equation}\label{defdeV2}
V=\{v; \ v_\alpha  \in H^1(\Omega ), \ v_3 \in H^2(\Omega); \ v_{| \Gamma _0}
=0, \ \partial _\nu v_{3| \Gamma _0}
=0 \textrm{ in the sense of trace}\},
\end{equation}
where $\nu $ is the normal to $\Gamma _0$ (i.e. the normal to the boundary which lies in the tangent plane), which corresponds to the clamped case.
\end{remark}

The following Lemma was obtained by Bernardou and Ciarlet see \cite{Ciarlet}.

\begin{lemma}\label{CB}
The bilinear form $a+b$ is coercive on $V$.
\end{lemma}

We shall denote by $V'$ the dual space of $V$. Here dual is obviously understood in the abstract sense of the space of continuous linear functionals on $V$. In order to make explicit computations in terms of equation and boundary conditions, we shall often take $f$ as a "function" defined on $\Omega $, in the space
\begin{eqnarray}\label{4.5}
& & \{f \in H^{-1}(\Omega ; \R  )\times H^{-1}(\Omega ; \R  )\times H^{-2}(\Omega ; \R  ) ; \\
& &  \ f \textrm{ "smooth" in a neighbourhood
of  }\Gamma _1 \} \subset V', \nonumber
\end{eqnarray}
where "smooth" means allowing classical integration by parts.
Obviously other choices for $f$ are possible.

Moreover, we immediately obtain the following result.
\begin{proposition}\label{coerciviteCB}
For $\varepsilon >0$ and for $f $ in $V'$, the variational problem (\ref{pbvar}) is of Lax-Milgram type and
it is a self-adjoint problem which has a coerciveness constant larger than $c \varepsilon ^2$, with $c>0$.
\end{proposition}

\begin{remark}
It is to be noticed that the coerciveness of the previous problem disapears when $\varepsilon =0$.
\end{remark}


\section{The ellipticity of systems and the Shapiro-Lopatinskii condition}\label{sectionSL}
\noindent

In this section, we recall some classical results on the linear boundary value problems for elliptic systems in the sense of Douglis and Nirenberg \cite{DN}. We begin with the definition of ellipticity for systems, then we recall the Shapiro-Lopatinskii
condition. This latter condition states which boundary conditions are well suited in order to have well posed problems for elliptic systems. We then recall in what sense an elliptic system with  Shapiro-Lopatinskii condition is "well-behaved". 

For brevity, from now on we will denote SL the Shapiro-Lopatinskii condition.

\subsection{Elliptic systems in the sense of Douglis and Nirenberg \cite{DN}}\label{sectionElliptic}
\noindent

In this work, we shall  deal with systems of $l$ ($l=3$ or $l=6$)
equations with $3$ unknowns (noted here $u_1,u_2,u_3$) defined on an open set $\Omega \subset \R^2$ with smooth boundary, which has the form:
\begin{equation}\label{system}
l_{kj}u_j=f_k , \  k=1, \dots ,  l,
\end{equation}
or equivently $L\textbf{u}=f$. 
The coefficients $l_{kj}(x, D )$ with $D = (D_1, D_2)$ and $D_l = -i \frac{\partial }{\partial x _l}$, $l \in \{1,2\}$,  are linear differential operators with real smooth coefficients. In our systems (\ref{system}),  the highest order of differentiation is different for the three unknowns and depends on the equation.  A way to take into account such differences between the various equations and unknowns is to define integer indices $(s_1,s_2,s_3)$ attached to the equations and integer indices $(t_1,t_2,t_3)$ attached to the unknowns (see Douglis and Nirenberg \cite{DN}) so that the "higher order terms" (which will be called "principal terms") are in equation $j$ the terms where each unknown "k" appears by its derivative of order $s_k+t_j$.
More precisely,  the  integers $(s_k, t_j)$ are such that 
$$
\left\{\begin{array}{l}\textrm{if } s_k+t_j  \ge 0, \textrm{ the order of } l_{kj} \textrm{ is less or equal to }s_k+t_j,\\
\textrm{if } s_k+t_j <0, \ l_{kj} \textrm{ is equal to zero.}  
\end{array}\right.
$$
The \textit{principal part} $l'_{kj}$ of $l_{kj}$ is obtained by keeping the terms of order $s_k+t_j$ if $s_k+t_j \ge 0$ and by taking $l_{kj}'=0$ if $s_k+t_j <0$.  The  matrix $L'(x, \xi)$, $\xi=(\xi_1, \xi_2)\in \R^2$, obtained by substituting $ \xi _\alpha $ for $D _\alpha $ in $l'_{kj}$, is called the \textit{principal symbol of the system}. Since $l'_{kj}$ are homogeneous of order $s_k+t_j$ with respect to $\xi _\alpha$, the determinant of the matrix $L'(x,\xi)$, denoted $\textrm{D}(x, \xi)$,  is  homogeneous of degree $\Sigma _k s_k +\Sigma _j t_j$. 
\begin{definition}
The system (\ref{system}) is \textit{elliptic in the sense of Douglis and Nirenberg} at the point $x \in \Omega$ if and only if 
\begin{equation}\label{conditionelli}
\textrm{D}(x, \xi) \neq 0, \ \ \forall \xi \in \R^2\setminus \{0\}.
\end{equation}
\end{definition}
\begin{remark}
Since the coefficients are assumed to be real, the function $\textrm{D}(x,\xi)$ for an elliptic system is even in $\xi$ of order $2m$ with 
$$\Sigma _k s_k +\Sigma _j t_j=2m.$$
\end{remark} 
\begin{remark}
The definition of the indices $s_j$ and $t_k$ for a system is slightly ambiguous. Indeed the result is exactly the same after adding an integer $n$ to the indices $s_j$ and substracting $n$ from the $t_k$.
\end{remark}
\begin{remark}\label{ellipticityremark}
Let $x_0  \in  \Omega $ be such that the system (\ref{system}) is not elliptic, then there exists a $\xi \in \R^2\setminus \{0\}$ such that $\textrm{D}(x_0, \xi )=0$. In such a case the system $L'(x_0, D ) u=0$, with frozen coefficients at $x_0$ admits a solution of the form $u(x)=v e^{i\xi x}$, with  $v \in \R ^3\setminus \{0\}$. 
\end{remark}

\begin{remark}
Moreover, throughout this paper, \textit{ellipticity} will be understood in the sequel as uniform, i.e.  there exists a positive constant $A$ such that 
$$A^{-1} \Sigma _{\alpha } | \xi_\alpha  |^2 \le   |\det L'(x, \xi)  |\le A \Sigma _\alpha   |\xi_\alpha   |^2,$$
for all $x\in \Omega $ and $\xi =(\xi_1, \xi _2) \in \R^2$. 
\end{remark}

\subsection{Shapiro-Lopatinskii conditions for elliptic systems in the sense of Douglis and Nirenberg \cite{DN}}\label{sectionSL3}
\noindent

From now on, for simplicity,  we will say that a system is \textit{elliptic} when it is elliptic  in the sense of Douglis and Nirenberg \cite{DN}.

Let $l_{kj}$ ($L$) be an elliptic system of order $2m$ with principal part $l'_{kj}$ ($L'$) and let $m$ boundary conditions be given by:
$$b_{kj} u_j=g_k, \ k \in \{1,..,m\},$$
where $b_{kj}(x, D )$ are differential operators with smooth coefficients. Let us define the integers $r_k$ (indices of the boundary conditions, $k=1,..,m$) such that 
$$
\left\{\begin{array}{l} \textrm{if }r_k+t_j\ge 0, \textrm{ the order of } b_{kj} \textrm{ is less or equal to } r_k+t_j \\
\textrm{if } r_k+t_j <0, \ b_{kj} \textrm{ is equal to zero.}  
\end{array}\right.
$$
The principal part $b'_{kj}$ is $b_{kj}$ if $r_k+t_j \ge 0$ and zero otherwise.

Assume that the smooth real coefficients are defined in
$\overline{ \Omega }$.

Let $x_0 \in \Gamma $, we assume that $L'$ is elliptic 
at $x_0$. Usually, see \cite{ADN} and \cite{Egorov} for instance, the SL condition at $x_0$ is defined via a local diffeomrophism sending a neighbourhood of $x_0$ in $\Omega $ into a neighbourhood of the origin in a half-plane. For ulterior computations, it is worth-while to take a special diffeomorphism which amounts to taking locally cartesian coordinates $x_1,x_2$, respectively, tangent and (inwards) normal to the boundary at $x_0$. We then consider only the principal parts of the equations and of the boundary conditions frozen at $x_0$. 
Next, we consider the corresponding boundary value problem obtained by formal tangential
Fourier transform (i.e. $D_1 \to  \xi_1$, with $\xi_1 \in \R$ and $u \to \tilde{u}$) which amounts to the following algebraic conditions:
\begin{equation}\label{SL3}
\left\{\begin{array}{l}\tilde{l}'_{kj}(x_0,\xi_1, D_2)\tilde{u}=0 \textrm{ for } x_2>0 \\
\tilde{b}'_{kj}(x_0,\xi_1, D_2)\tilde{u}=\tilde{g}_j 
\textrm{ for } x_2=0, 
\end{array}\right.
\end{equation}
$ j,k \in \{1,...,m\}$, see \cite{EgorovMeunierSanchez} Sec. 3.2 for details, if necessary.

The problem (\ref{SL3}) involves a system of ordinary differential equations with constant coefficients of the variable $x_2 \in \R^+$ and $m$ boundary conditions at $x_2=0$, whose solutions are classically a linear combination of terms of the form:
\begin{equation}\label{jordan}
\tilde{u}(\xi _1,x_2)=\left\{\begin{array}{l} v e^{i \xi _2 x_2}, \ v \in \C ^3 \\
P(x_2)  e^{i \xi _2 x_2},\textrm{  where } P\textrm{  is a polynomial, in the case of Jordan block. } 
\end{array}\right. 
\end{equation}

Recalling that the system $L$ is elliptic, it follows that 
the imaginary part of $\xi _2$ does not vanish. Furthermore,  
there are $m$  solutions $\xi_2$ of  $D(x_0, \xi_1, \xi_2)=0$ with positive imaginary part 
that we denote $\xi_2^+$ (and $m$ with negative imaginary part denoted $\xi_2^-$).  

We then try to solve (\ref{SL3}) using only linear combinations of the $m$ solutions of the form (\ref{jordan}) for the $m$ roots $\xi _2 ^+$ (i.e. exponentially decreasing towards the domain).

\begin{definition}
The SL condition is satisfied at $x_0 \in \Gamma$ if  one of the following equivalent conditions holds:
\begin{enumerate}
\item The solution of the previous problem
is defined uniquely.
\item Zero is the only solution of the homogeneous (i.e. with $g_j=0$) 
previous problem.
\end{enumerate}
\end{definition}
\begin{remark}
The two conditions (which are equivalent) of the previous definition are clearly equivalent to the non annulation of the determinant of the corresponding algebraic "system".
\end{remark}
\begin{remark}
The reason for defining the SL condition amounts to the possibilty of solving the problem in a half plane via tangential Fourier transform. The reason for not considering the $\xi_2^-$ roots is that, for $x_2>0$, they should give exponentially growing Fourier transforms in $x_1\to \xi_1$, which are not allowed in distribution theory (note that $\xi _1$ and $\xi _2$ are proportional as $D(\xi)$ is homogeneous).
\end{remark}

The verification of the SL condition is often tricky. In some situations, we can use equivalent expressions which are simpler to treat. More precisely, define the function 
$u$ by $u(x_1,x_2)=\tilde{u}(\xi _1,x_2)e^{i |\xi _1| x_1}$, with $\tilde{u}(\xi _1,x_2) =v e^{i \xi _2^+ x_2} $ (or  expressed as exponential polynomial in the case of Jordan block), it  is an exponentially decreasing function in the direction inwards the domain (when $x_2 \to + \infty$), it is also a periodic function in the tangential direction $x_1$ and it satisfies 
\begin{equation}\label{SL3bis}
\left\{\begin{array}{l}\tilde{l}'_{kj}(x_0,D_1, D_2)u=0 \textrm{ for } x_2>0 \\
\tilde{b }'_{kj}(x_0,D_1, D_2)u=g_j 
\textrm{ for } x_2=0, 
\end{array}\right.
\end{equation}
$ j,k \in \{1,...,m\}$.
The following proposition is very useful in the case where ellipticity is linked with positive energy integrals obtained by integrating by parts. For instance, we have:
\begin{proposition}
Consider the homogeneous problem associated with (\ref{SL3bis}) (i.e. taking $g_j=0$) for $x_0 \in \Gamma$. If any solution $u$, which is periodic in the tangential direction $x_1$ and exponentially decreasing in the direction $x_2$ inwards the domain, is zero,  then the SL condition is satisfied. 
\end{proposition}

\begin{remark}
In order to have well-posed problems for elliptic systems, boundary conditions satisfying the SL condition should be prescribed at any points of the boundary.  Their number  is half the total order of the system. 
\end{remark}
\begin{remark}
The specific boundary conditions may differ from a point to another on the boundary. In particular, each connected component of the boundary may have its own set of boundary conditions. Otherwise, local changes of boundary conditions (as well as non-smoothness of the boundary) induces local singularities. 
\textbf{A changer}
\end{remark}

\subsection{Some results for "well posed" elliptic systems}
\noindent

Let us now consider a boundary value problem formed by an elliptic system with boundary conditions satisfying the SL condition. In what sense is this problem "well-behaved"? The obvious example of an eigenvalue problem, even for an equation shows that uniqueness is only ensured up to the kernel formed by the eigenvectors associated with the zero eigenvalue, whereas existence involves compatibility conditions (orthogonality to the kernel of the adjoint problem). The general results are those of  Agmon, Douglis and Nirenberg  \cite{ADN}.

First, let us recall the definition of a Fredholm operator.
\begin{definition}\label{defopindice}
Let $E$ and $F$ be two Hilbert spaces and $A$ an operator (closed with dense domain in $E$) from $E$ into $F$. We say that $A$ is a Fredholm operator if and only if the following three conditions hold:
\begin{enumerate}
\item $\textrm{Ker}(A) $ is of finite dimension,
\item $\textrm{R}(A)$ is closed,
\item $\textrm{R}(A)$ is of finite codimension.
\end{enumerate}
The operator $A$ is also said to be an index operator, the index is defined as $\textrm{dim Ker}(A)-\textrm{codim R}(A)$.
\end{definition}

Let us consider an elliptic system of order $2m$ whose coefficients are smooth:
\begin{equation}\label{systemgeneral}
\left\{\begin{array}{l}l_{kj} u_j=f_k, \ j,k \in \{1,...,l\} \textrm{ in } 
\Omega \\
b_{hj} u_j=g_h, \ h \in \{1,...,m\} \textrm{ on } \partial \Omega ,
\end{array}\right.
\end{equation} 
whose indices associated with unknowns, equations and boundary conditions are respectively  $t_j,s_j, r_j$. 
Let $\rho$  be a "big enough" 
real number, called regularity index. Consider operator (\ref{systemgeneral}) as a linear operator from the space $E$ to the space $F$ defined by:
\begin{equation}\label{defespaces}
E= \Pi_{j=1}^l H^{\rho +t_j}(\Omega), \ F= \Pi_{j=1}^l H^{\rho -s_j}(\Omega)\times  \Pi_{j=1}^m H^{\rho -r_j-\frac{1}{2}}(\partial \Omega).
\end{equation}
The real $\rho$ is chosen in order to give a sense to the traces which are involved, i.e. it is  such that $\rho -r_j-1/2>0$ for $j \in \{1,...,m\}$.

The following result is the main result of the theory of Agmon, Douglis and Nirenberg:
\begin{theorem}[Agmon, Douglis and Nirenberg  \cite{ADN}]\label{ThADN}
Let $\Omega $ be a bounded open set with smooth boundary $\Gamma$. Let us consider an elliptic system with boundary conditions satisfying the SL condition everywhere on $\Gamma $.
Assume that the coefficients of the system are smooth and that $u,f$ and $g$ satisfy (\ref{systemgeneral}). Then the following estimate holds true:
\begin{equation}\label{ADNinegalite}
\Norm{u}_E \le C(\Norm{(f,g)}_F +\Norm{u}_{(L^2(\Omega ))^l}),
\end{equation}
where $C$ does not depend on $u,f,g$. Moreover, the operator defined by (\ref{systemgeneral}) from the space $E$ to the space $F$, given by (\ref{defespaces}), is a Fredholm operator, for all value of $\rho$ such that  $\rho -r_j-1/2>0$ for $j \in \{1,...,m\}$. Furthermore, the dimension of the kernel and the dimension of the subspace  orthogonal to the range do not depend on $\rho$. The kernel is composed of smooth functions.
\end{theorem}
\begin{remark}\label{remarkADN}
The previous theorem means that in general existence and uniqueness of the solution only hold up to a finite number of compatibility conditions for $f$ and $g$ and existence of the solution holds up to a finite dimension kernel. More precise properties need specific properties of the system. 
\end{remark}
\begin{remark}\label{noyau}
For all values of $\rho $, the kernel formed by the eigenvectors corresponding to the eigenvalue $0$ is of finite dimension and is composed of smooth functions, independent of $\rho$ (in $\mathcal{C}^\infty (\overline{\Omega})$). 
\end{remark}

\begin{remark}\label{remarkADN2}
Denote $A$ the operator defined by (\ref{systemgeneral}) in the spaces $E$ and $F$. Let us consider the case where $\textrm{dim Ker}(A) >0$ and
define the inverse $B$ of $A$ as a closed operator from $\textrm{R}(A)$ to  $E/\textrm{Ker}(A)$, we have that
\begin{equation}\label{estimation}
\Norm{\tilde{u}}_{E/\textrm{Ker}(A)}\le C \Norm{(f,g)}_F,
\end{equation}
where $\tilde{u}$ is an element of the equivalence class of $u$.

The element $\tilde{u}$ can also be viewed as an element of the orthogonal of $\textrm{Ker}(A)$ in $E$, which is identified with $E/\textrm{Ker}(A)$. In such a case, there exists a unique $(\tilde{u}, \hat{u}) \in  E/\textrm{Ker}(A) \times \textrm{Ker}(A)$ such that 
$$u=\tilde{u}+ \hat{u}.$$
Since $\textrm{Ker}(A)$ is of finite dimension, all the norms are equivalent and we can choose for $\hat{u}$ a norm in a space $H^{-\nu}$ with $\nu$ very big. Therefore, inequality (\ref{ADNinegalite}) can be rewritten as 
\begin{equation}\label{ADNinegalite2bis}
\Norm{u}_E \le C(\Norm{\tilde{u}}_{E/\textrm{Ker}(A)} +\Norm{\hat{u}}_{H^{-\nu}})\le C(\Norm{\tilde{u}}_{E/\textrm{Ker}(A)} +\Norm{u}_{H^{-\nu}}),
\end{equation}
for $\nu $  big enough such that $E \subset H^{-\nu}$. Recalling (\ref{estimation}), we then deduce that 
\begin{equation}\label{ADNinegalite2}
\Norm{u}_E \le C(\Norm{(f,g)}_F +\Norm{\hat{u}}_{H^{-\nu}}).
\end{equation}
Moreover, the norm in $H^{-\nu}$ may be replaced by a seminorm, provided it is a norm on Ker($A$).
\end{remark}
\begin{remark}\label{remarkADN3}
In the case where $\textrm{dim Ker}(A) =0$, the inverse $B$ of the operator  $A$ is well defined on $R(A)$. It is a closed operator, hence it is bounded and the following estimate holds:
\begin{equation}\label{ADNinegalite3}
\Norm{u}_E \le C\Norm{(f,g)}_F .
\end{equation}
\end{remark}

\section{Study of four systems involved in shell theory}\label{4pb}
\noindent

In this section, we study four systems, denoted by \textit{rigidity} system, \textit{membrane tension} system, \textit{membrane} system and \textit{Koiter shell} system, which will appear in the sequel. We prove that these four systems satisfy the ellipticity condition and we study some boundary conditions. It is to be noticed that the boundary conditions may be different on $\Gamma_0$ and $\Gamma _1$ which are supposed to be disjoints.

Let us recall the situation: $\Omega $ is a connected bounded open set of $\R^2$ with $\mathcal{C}^\infty$ boundary $\Gamma = \Gamma_0 \cup \Gamma _1$ and $\Gamma _0 \cap \Gamma _1 =\emptyset $. The middle surface $S$ of the shell is the image in $\mathcal{E}^3$ of $\Omega $ for the map 
$$\varphi \ : \ (y^1,y^2) \in \overline{\Omega} \to \varphi (y) \in \mathcal{E}^3.$$
We assume that the ellipticity assumption of the surface holds:
\begin{equation*}\label{elliptic2}
b_{11}b_{22}-b_{12}^2 >0 \textrm{ uniformly on } \Omega.
\end{equation*}

\subsection{The rigidity system}
\noindent

Let us begin with the \textit{rigidity} system defined by  
$\gamma_{\alpha \beta }(u)$:
\begin{equation}\label{rigidity1}
\left\{\begin{array}{l}\gamma _{11 } (u):=  \partial _1u _{1}-\Gamma _{11}^\alpha u_\alpha-b_{11} u_3\\
\gamma _{22 } (u):=\partial _2u _{2}-\Gamma _{22}^\alpha u_\alpha-b_{22} u_3 \\
\gamma _{12 } (u):= \frac{1}{2} ( \partial _2u _{1}+\partial _1 u _{2})-\Gamma _{12}^\alpha u_\alpha-b_{12} u_3.\end{array}\right.
\end{equation}
Clearly $u_\alpha $ and $u_3$ play very different roles as $u_\alpha $ appears with derivatives whereas $u_3$ only appears without. Therefore take $(1,1,0)$ as the indices of the unknowns $(u_1,u_2,u_3)$ and $(0,0,0)$ as equation indices in the order $(\gamma _{11 },\gamma _{22 },\gamma _{12 })$. The principal system is obtained by substituting $0$ for $\Gamma _{\lambda \mu}^\alpha$ but keeping $b_{\lambda \mu }$.

\begin{lemma}\label{DNrigidite}
Do to the ellipticity assumption of the surface (\ref{elliptic}), the \textit{rigidity} system $\gamma$ 
is elliptic of total order $2$ on $\Omega $.
\end{lemma}
\begin{proof} 
Substitute $-i \xi _\alpha$ for $ \partial _\alpha $  in the principal system, we obtain a system whose determinant is  $D(x,\xi )=2b_{12}\xi_1 \xi_2-b_{22}\xi _1^2 -b_{11} \xi_2^2 $, hence due to  the ellipticity hypothesis (\ref{elliptic}), for all $x \in \Omega $, we have
$$D(x,\xi )>0. $$
\end{proof}

\subsubsection{Cauchy boundary conditions}
\noindent 

It is classical that the Cauchy problem associated with elliptic system is not well posed in the sense that it does not enjoy existence, uniqueness and stability of solutions. Nevertheless, the Cauchy problem associated with the rigidity system will be involved in the sequel and we study it now. In particular, we shall need the following uniqueness theorem for solutions $u \in H^1\times H^1\times L^2$.

\begin{lemma}\label{rigide}
Under the ellipticity assumption of the surface (\ref{elliptic}), the system $\gamma _{\alpha \beta } (u)=0$ on $\Omega$ with the boundary conditions $u_1=u_2=0$ on a part of the boundary (of positive measure) admits a unique solution which is $u=0$. 
\end{lemma}
\begin{proof}
Let us assume that $v\in H^1(\Omega )\times H^1(\Omega) \times H^2(\Omega) $ is such that  $ \gamma _{\alpha \beta}(v)=0$ and $v_1=v_2=0$ 
on a part of the boundary.
Thanks to the ellipticity hypothesis (\ref{elliptic}), we know that 
$b_{11}\neq 0$ 
on $\overline{\Omega}$. We can eliminate $v_3$ from the first and third equations ($\gamma _{11}(v)=0$ and $\gamma_{22}(v)=0$) of the system $\gamma$. This yields the system of two equations for two unknowns $(v_1,v_2)$:
\begin{equation}\label{syst222}
\left\{\begin{array}{l}
0=\partial _2v _{2}-\Gamma _{22}^\alpha v_\alpha-\frac{b_{22} }{b_{11}}( \partial _1v _{1}-\Gamma _{11}^\alpha v_\alpha) \\
0= \frac{1}{2} ( \partial _2v _{1}+\partial _1 v _{2})-\Gamma _{12}^\alpha v_\alpha
-\frac{b _{12}}{b_{11}}( \partial _1v _{1}-\Gamma _{11}^\alpha v_\alpha)
.\end{array}\right.
\end{equation}
The eliminated unknown being then given by:
$$v_3=  \frac{1}{b_{11}}( \partial _1v _{1}-\Gamma _{11}^\alpha v_\alpha).$$
The problem then reduces to the uniqueness in the class $H^1(\Omega)$ of $(v_1,v_2)$ satisfying 
\begin{equation}\label{gamma33bis0}
\left\{\begin{array}{l}  \partial _1v _{1}-b_{11} v_3=0\\
\partial _2v _{2}-b_{22} v_3 =0\\
\frac{1}{2} ( \partial _2v _{1}+\partial _1 v _{2})-b_{12} v_3=0,\end{array}\right.
\end{equation}
with $v_1=v_2=0$ on a part of the boundary. This problem is more or less classical. Under analyticity hypotheses about the coefficients and the boundary, the uniqueness follows from Holmgren local uniqueness theorem and analytic continuation (as $u_1,u_2$ are in this case analytic inside $\Omega$). Under the $\mathcal{C}^\infty$ hypotheses adopted here, uniqueness follows from theory of pseudo-analytic functions. There are two nearly equivalent theories of such functions attached to the names of L. Bers (see for instance supplement of chapter IV of  \cite{CouHil}, written by Bers himself) and I.N. Vekua see \cite{Vekua}.   

Let $(v_1,v_2)$ be a solution of (\ref{gamma33bis0}) vanishing on a part $\Gamma$ of the boundary. Let $(\tilde{v}_1,\tilde{v}_2)$ be an extension of $(v_1,v_2)$ with values zero to an extended domain across $\Gamma$. Classically $(\tilde{v}_1,\tilde{v}_2)$ satisfies the same system (\ref{gamma33bis0}) on the extended domain and, according to interior regularity theory for elliptic systems, is of class $\mathcal{C}^\infty$ inside it. The function $\tilde{w}=\tilde{v}_1+i\tilde{v}_2$ is pseudo-analytic, of class $\mathcal{C}^\infty$ and vanishes on the outer region of the extended domain. We then use either theorem 3.5 of \cite{Vekua}, p. 146, which gives directly the uniqueness or the representation theorem of \cite{CouHil} p. 379. In this case, $\tilde{w}(z)$ admits the expression (here $z=x_1+ix_2$):
$$\tilde{w}(z)=e^{\delta (z)}f(z),$$  
where $f(z)$ is analytic and $\delta (z)$ is continuous. As $e^{\delta (z)}$ vanishes nowhere, the uniqueness follows.
\end{proof}

\begin{remark}
Strictly speaking, the evoked theorems of pseudo-analytic functions apply to systems with principal part of the canonical form
\begin{equation}\label{gamma33bis1}
 \left\{\begin{array}{l}  \partial _1v _{1}-\partial _2  v_2=...\\
\partial _2v _{1}+\partial_1 v_2 =...,\end{array}\right.
\end{equation}
so that the classical reduction to this form (see for instance \cite{CouHil} p. 169-170) should be previously considered. But obviously, this does not modify the $\mathcal{C}^\infty$ regularity inside the domain.
\end{remark}
Let us make several comments about this uniqueness result.
\begin{remark}
This result, known as the infinitesimal rigidity of the surface, 
does not depend on the curvilinear coordinates.
\end{remark}
\begin{remark}\label{instabilite}
The key ingredients of the previous uniqueness result 
are a uniqueness theorem for the Cauchy problem for elliptic systems of two equations of order 1.
It is not based upon a coercivity assumption for an elliptic system. 
But we know that the Cauchy problem for elliptic systems is precarious in the sense that it does not enjoy existence, uniqueness and stability of solutions. This means that such a system could lead to instability in the sense that there could exist $v_1,v_2, v_3$ very "big" in usual spaces such that $\gamma _{\alpha \beta}(v)$ are very "small".
\end{remark}

\subsubsection{Boundary value problems for the \textit{rigidity} system}
\noindent

From now on, we will consider the frame $(O,a_1,a_2,a_3)$ to be orthonormal on the boundary and such that $u_t=(u_1,0,0)$ and $u_n=(0,u_2,0)$, where $u_t$ denotes the component of $u$ in the tangential direction to the boundary and $u_n$ is the component of $u$ in the normal direction to the boundary and in the tangent plane. This point which is not absolutely necessary, implies a special local parametrization.
\begin{lemma}\label{DNdirichlet}
The  boundary condition $u_1=g$
satisfies the SL condition for the system $\gamma$. 
\end{lemma} 
\begin{proof}
We take as index of the boundary condition $r=-1$. 
Let $x_0$ belong to $\Gamma $. As explained in Section \ref{sectionSL3}, using a partition of unity, local mappings, with axes $y_1$ tangential and $y_2$  inwards $\Gamma $, dropping lower
order differential terms, we obtain a new system:
\begin{equation}\label{gamma33bis}
\textrm{For }y_2>0, \ \left\{\begin{array}{l}  \partial _1u _{1}-b_{11} u_3=0\\
\partial _2u _{2}-b_{22} u_3 =0\\
\frac{1}{2} ( \partial _2u _{1}+\partial _1 u _{2})-b_{12} u_3=0.\end{array}\right.
\end{equation}
We look for solutions which are exponentially decreasing when $y_2 \to +\infty$ of the form:
$$u(y_1,y_2)= U e^{i \zeta y_2+i \xi _1 y_1}, \ \xi_1 \in \R\setminus \{0\},$$
with $U = \left(\begin{array}{l} U_1\\ U_2 \\ U_3\end{array}\right)\in \C ^3$, $\textrm{Im}(\zeta ) >0$.
Substituting this solution into (\ref{gamma33bis}) and
using the boundary condition we have $U_1=0$. Consequently, $u_1=0$  everywhere and (\ref{gamma33bis}) gives also $u_2=u_3=0$.
$U_2=U_3=0$. 
\end{proof}
\begin{remark}\label{bijectionajout}
Similarly to the proof of the previous result, we can prove that the following boundary conditions satisfy the SL condition:
\begin{enumerate}
\item $u_2=g$ (take $r=-1$). 
\item $u_3=g$ (take $r=0$). 
\end{enumerate}
\end{remark}

\begin{remark}\label{resultatutile}
Since $\Gamma _0$ and $\Gamma _1$ are disjoints and thanks to the previous statements, the boundary value problem 
\begin{equation}\label{pbbienpose}
 \left\{\begin{array}{l}  
 \gamma _{\alpha \beta} (u)=0 \textrm{ on } \Omega ,\\
u _2=0 \textrm{ on } \Gamma _0 ,\\
u_3=\tilde{u} \textrm{ on } \Gamma _1.\end{array}\right.
\end{equation}
is  "well posed" in the Agmon, Douglis and Nirenberg sense. Recalling Theorem \ref{ThADN} and Remark \ref{remarkADN}, together with standard regularity theory for elliptic systems, it follows that $u$ is of class $C^{\infty}$ on $\Omega \cup \Gamma _0$ for any $\tilde{u}$ (either smooth or not). Consequently, up to a kernel of  finite dimension composed of smooth functions belonging to $\mathcal{C}^\infty (\overline{\Omega })^3$ (and eventually up to a compatibility condition (to belong to the range of the operator which is a closed subspace of finite codimension), the space $\{v, \ \gamma_{\alpha \beta }(v)=0 \textrm{ on } \Omega, \ v_n =0 \textrm{ on } \Gamma _0\}$ is isomorphic with the space   $\mathcal{C}^\infty(\Gamma _1)$. The previous statements can be rephrased as follows: up to a  finite dimensional space composed of smooth functions,  the space
$\{v, \ \gamma_{\alpha \beta }(v)=0 \textrm{ on } \Omega, \ v_n =0 \textrm{ on } \Gamma _0\}$ is isomorphic to the space of traces on $\Gamma _1$:
\begin{equation}
\{\tilde{v} \in C^{\infty }(\Gamma _1)\},
\end{equation}
the isomorphism is obtained by solving (\ref{gamma33bis}).

In the sequel, we shall consider indifferently the functions $v $ (defined up to an additive element of the kernel) or their traces $\tilde{v}$ on
$\Gamma _1$.
\end{remark}

\subsection{The system 
of \textit{membrane tensions}}
\noindent

Consider the \textit{membrane tensions} system $\mathcal{T}$ of three equations with the three unknowns $(T^{11}, T^{22},T^{12})$:
\begin{equation}\label{Tbis}
\left\{\begin{array}{l}
 -T^{11 }_{|1}-T^{21 }_{|2}=f^1\\
-T^{22 }_{|2}-T^{21 }_{|1}=f^2\\
-b _{11}T^{11}-2b _{12} T^{12}-  b _{22}T^{22}=f^3.\end{array}\right.
\end{equation}
It is apparent that the three unknowns play analagous roles. Concerning the equations, it is clear that the first and the second are similar but different from the third. Therefore, we consider  $(1,1,0)$ as indices of equations and  $(0,0,0)$ as  indices of unknowns. The principal system $\mathcal{T}_P$ is obtained by replacing the covaraint differentiation $_{|_\alpha}$ by the usual differentiation $ \partial _\alpha$ (i.e. replacing $\Gamma _{\alpha \beta}^\lambda$ by zero). Proceeding as in the proof of Lemma \ref{DNrigidite}, we obtain the following result.
\begin{lemma}\label{DNmembranetensions}
Under the ellipticity assumption of the surface (\ref{elliptic}), the system $\mathcal{T}$ is elliptic of total order two. 
\end{lemma}
\begin{remark}\label{rigide2}
It is worthwhile to study the Cauchy problem for the membrane tension system (\ref{Tbis}). This is done exactly as in Section 4.1.2 for the rigidity system. We eliminate one of the unknowns, $T^{11}$ for instance and (\ref{Tbis}) reduces to an elliptic system of two first order equations in $T^{12}$ and $T^{22}$. The Cauchy conditions are $T^{12}=T^{22}=0$ on a part of the boundary. According to our special frame, this amounts to $T^{\alpha \beta}n_\beta =0$. This Cauchy problem enjoys uniqueness but not existence and stability in usual spaces. 
\end{remark}

\begin{remark}
The system of membrane tensions $\mathcal{T}$ (\ref{Tbis}) and the system of rigidity $\gamma $ (\ref{rigidity1}) are adjoint to each other. This is easily checked by covariant integration by parts on $S$. Indeed, neglecting boundary terms (we are only interested in the equations) and using (\ref{gamma}) together with the symmetry of the $T^{\alpha \beta})$, we have:
\begin{eqnarray*}\label{intpartie1}
\int_{S} T^{\alpha \beta }\gamma _{\alpha \beta } (u) \D s &=&\int_{S} T^{\alpha \beta }\Big(\frac{1}{2}(u_{\alpha | \beta}+u_{\beta | \alpha}) -b _{\alpha \beta } u_3\Big) \D s \\
&=&\int_{S} T^{\alpha \beta }\Big(u_{\alpha | \beta} -b _{\alpha \beta } u_3\Big) \D s \\
&=&-\int_{S} \Big(T^{\alpha \beta }_{| \beta }u_{\alpha }+T^{\alpha \beta } b _{\alpha \beta } u_3\Big) \D s \\
&=&\int_{S} \mathcal{T}(T)u \D s
\end{eqnarray*}
\end{remark}


\subsection{The \textit{membrane} system}\label{membranesystem}
\noindent

We denote by \textit{membrane} system the system of three equations with three unknowns $u=(u_1,u_2,u_3)$ obtained from (\ref{Tbis}) when the tensions are written in terms of $u$, i.e. 
\begin{equation}\label{Tbisbis}
\left\{\begin{array}{l}
 -T^{11 }_{|1}(u)-T^{21 }_{|2}(u)=f^1\\
-T^{22 }_{|2}(u)-T^{21 }_{|1}(u) =f^2\\
-b _{11}T^{11}(u)-2b _{12} T^{12}(u)-  b _{22}T^{22}(u)=f^3,\end{array}\right.
\end{equation}
with 
\begin{equation}\label{defdeAbis}
T^{\alpha \beta }(u)= A^{\alpha \beta \lambda \mu} \gamma _{\lambda \mu } (u),
\end{equation}
and
\begin{equation}\label{derdeTbis}
 T^{\alpha \beta }_{|k}(u)=\partial _kT^{\alpha \beta}(u)+\Gamma _{k n}^ \beta T^{\alpha n}(u)+\Gamma _{k m}^ \alpha T^{\beta m}(u).
\end{equation}
In order to prove the ellipticity of the \textit{membrane} system, we replace it by another, equivalent one. Indeed, we shall take as unknowns $u_1,u_2,u_3$ and the supplementary auxiliary unknowns $T^{11}, T^{22}, T^{12}$. Inverting the matrix $A^{\alpha \beta \lambda \mu }$  in (\ref{defdeAbis}) and recalling the definition of $\gamma$, we obtain the following equivalent system:
\begin{equation}\label{systemeajout1}
\left\{\begin{array}{l}
 -T^{11 }_{|1}-T^{21 }_{|2}=f^1\\
-T^{22 }_{|2}-T^{21 }_{|1}=f^2\\
-b _{11}T^{11}-2b _{12} T^{12}-  b _{22}T^{22}=f^3,\end{array}\right.
\end{equation}
\begin{equation}\label{systemeajout2}
\left\{\begin{array}{l}
 u_{1|1}-b_{11}u_3-C_{11\alpha \beta} T^{\alpha \beta }=0\\
 u_{2|2}-b_{22}u_3-C_{22\alpha \beta} T^{\alpha \beta }=0\\
 \frac{1}{2}(u_{1|2}+u_{2|1})-b_{12}u_3-C_{12\alpha \beta} T^{\alpha \beta }=0,
 \end{array}\right.
\end{equation} 
where $C_{\alpha \beta \lambda \mu} $ are the \textit{compliances} (inverse matrix of $A^{\alpha \beta \lambda \mu} $). The system (\ref{systemeajout1}) and (\ref{systemeajout2}) is a system of six equations with the six unknowns  $(T^{11}, T^{22}, T^{12}, u_1, u_2,u_3)$ (written in this order). We recognize the \textit{membrane tension} system in (\ref{systemeajout1}) and the \textit{rigidity} system in (\ref{systemeajout2}).
Consider $(1,1,0,0,0,0)$ as indices of equations  and  $(0,0,0,1,1,0)$ as indices of unknowns. Then  replacing the  differentiation $ \partial _\alpha$ by $-i \xi _\alpha$ and taking the determinant of the obtained system, we have a determinant of the form 
$$ \left| \begin{array}{ll}
 D_{11}&0\\
 D_{21}& D_{22}
 \end{array}\right|= 0=\left| \begin{array}{l}
 D_{11}
 \end{array}\right| \left| \begin{array}{l}
  D_{22}
 \end{array}\right|,$$
where the $D_{\alpha \beta}$ are $3 \times 3$ matrices. Moreover, $D_{11} $ and $D_{22}$ are precisely those of the \textit{membrane tension} system  and the \textit{rigidity} system respectively and ellipticity follows. The same result is obviously obtained without using the auxiliary unknowns $T^{\alpha \beta}$, in fact, we have,
\begin{lemma}\label{DNmembrane}
Under the ellipticity assumption of the surface (\ref{elliptic}), the \textit{membrane} system with indices (of unknowns and of equations) $(1,1,0)$ $(1,1,0)$ is elliptic of total order four.
\end{lemma}

Let us now state boundary value problems which will be considered later on. It is to be noticed that only two boundary conditions are considered on $\Gamma _0$.

\begin{proposition}\label{shapirolopatinskii1}
The boundary value problem  
\begin{equation}\label{pbvar3bispreambule}
\left\{\begin{array}{l}
  -\partial _1 T^{11}(u)-  \partial _2 T^{21}(u)=f^1\\
  -\partial _2 T^{22}(u)-  \partial _1 T^{21}(u) =f^2\\
-b _{11}T^{11}(u)-2b _{12} T^{12}(u)-  b _{22}T^{22}(u)=f^3 
\\
u_1= u_2 =0, \textrm{ on }\Gamma _0\\
T^{\alpha \beta} (u)n_\alpha =0 \textrm{ on } \Gamma _1, \, \beta \in \{1,2\} .
\end{array}\right.
\end{equation} 
with unknown $u$ satisfies the SL condition  on $\Gamma _0 $ but it does not on $\Gamma _1$.
\end{proposition}

\begin{remark}
The partial differential boundary value problem (\ref{pbvar3bispreambule}) is formally associated with the variational problem (\ref{pbvar}) when $\varepsilon =0$.
\end{remark}

\begin{proof}
Let us fix $x_0 \in \Gamma$. According to the definition of the SL
condition, we consider the homogeneous system with constant coefficients in which we only kept the principal terms, i.e. taking $\Gamma ^\lambda _{\alpha \beta}=0$ but $b_{\alpha \beta } \neq 0$ and $f^i=0$. 

After a change of coordinates with local mappings, still denoted
by $(x_1,x_2)$, we only have to consider solutions, which are
exponentially decreasing in the direction inwards the domain ($x_2$), of
the corresponding boundary value problem obtained by formal tangential
Fourier transform. Denoting by $\tilde{u}(\xi _1, x_2)$ such a
solution, by periodicity, we  can restrict the domain to the strip
$B=(0, 2\pi /| \xi _1|) \times (0, +\infty)$  and we can consider
the function 
\begin{equation}\label{defdev}
v( x_1, x_2)=e^{i \xi _1 x_1}  \tilde{u}( \xi _1,x_2),
\end{equation}
 which is periodic in the tangential direction $x_1$, decreasing as $x_2 \to + \infty $ and satisfies the homogeneous boundary condition associated with the principal part of (\ref{pbvar3bispreambule}).  
Recall that $v$ satisfies  the equation
\begin{equation}\label{bande01}
\left\{\begin{array}{l}
  -\partial _1 T^{11}(v)-  \partial _2 T^{21}(v)=0\\
  -\partial _2 T^{22}(v)-  \partial _1 T^{21}(v) =0\\
-b _{11}T^{11}(v)-2b _{12} T^{12}(v)-  b _{22}T^{22}(v)=0.
\end{array}\right.
\end{equation}
We  multiply each line of (\ref{bande01}) by the conjugate $\overline{v}_i$ and we integrate by parts on the periodicity layer  $B$. We see that on the infinite boundary the boundary integral is vanishing thanks to the decreasing condition as $x_2 \to + \infty $. The boundary integral also vanishes on the lateral boundary (which is parallel to $x_2$) of the strip thanks to the periodicity of $v$. Recalling the definition of $T^{ij}$, we obtain  
\begin{equation}\label{4.1}
\int _{B} A^{\alpha \beta \lambda \mu}\gamma _{\lambda \mu}(v) \gamma _{\alpha \beta }(\overline{v}) \D x_1 \D x_2  =0,
\end{equation}
where obviously all the $\Gamma ^\alpha _{\beta \gamma}=0$.
Consequently, recalling the positivity property (\ref{coercivity}) of $A$, this yields that  
\begin{equation}\label{4.1}
 \int _{B} \Sigma _{\alpha \beta}
 |\gamma _{\alpha \beta }(\overline{v}) |^2\D x_1 \D x_2 =0,
\end{equation}
and then 
\begin{equation}\label{4.21}
\gamma _{\alpha \beta}(v) =0 \textrm{ on } B.
\end{equation}
 We have now to distinguish two cases. 
 
If $x_0 \in \Gamma_0$, then reasonning as in Lemma \ref{rigide} (or merely as in Lemma \ref{DNdirichlet}), we deduce that $v_1=v_2=v_3=0$, which means that the SL condition is satisfied on $\Gamma_0$.

Let now $x_0 \in \Gamma_1$ and 
\begin{equation}\label{4.2}
\gamma _{\alpha \beta}(v) =0 \textrm{ on } B .
\end{equation}
Remembering the definition (\ref{defdev}) of $v$, this yields that  $\tilde{u}$ is a solution of the following system of ODE of order 2:
\begin{equation}\label{bande1}
\left\{\begin{array}{l}
  i\xi _1 \tilde{u}_1-  b _{11}\tilde{u}_3 =0\\
 \partial _2 \tilde{u}_2-  b _{22} \tilde{u}_3 =0\\
\frac{1}{2}(\partial  _2 \tilde{u}_1+i\xi _1 \tilde{u}_2)-b_{12}\tilde{u}_3=0.
\end{array}\right.
\end{equation}
Thanks to the fact that $b_{11}\neq 0 $ and $b_{22} \neq 0$ this can be rewritten as:
$$
\left\{\begin{array}{l}
  \tilde{u}_1=-i \frac{ b _{11}}{\xi_1}\tilde{u}_3 \\
    \tilde{u}_3 =\frac{1}{b_{22}}\partial _2 \tilde{u}_2\\
b_{11} \partial ^2_2 \tilde{u}_2 -2ib_{12}\xi_1 \partial _2\tilde{u}_2  -b_{22} \xi _1^2\tilde{u}_2=0. 
\end{array}\right.
$$
Recalling the ellipticity condition (\ref{elliptic}), we obtain after an easy computation that there exists a complex solution $\tilde{u}$, given by $\tilde{u}= we^{ \lambda _-x_2}$, where $w\neq 0$ and $\lambda _-$ is the root with negative real part of 
$$b_{11} \lambda ^2 -2ib_{12} \xi _1 \lambda -b_{22} \xi _1^2=0. $$
This means that there exists non zero $v$ which is
exponentially decreasing in the direction inwards the domain
$$v( \xi_1, x_2)=we^{i \xi _1 x_1}e^{ \lambda _-x_2} ,$$
with Re$(\lambda _-)<0$ such that 
$$\gamma _{\alpha \beta}(v) =0 \textrm{ on } B ,$$
and hence 
$$T^{\alpha \beta} (v)n_\alpha =0 \textrm{ on } \Gamma _1 .$$ 
Therefore, the SL
condition is not satisfied on $\Gamma _1$.
\end{proof}

\subsection{The \textit{Koiter shell} system}
\noindent

The boundary value problem associated with the variational problem (\ref{pbvar}) with $\varepsilon >0$ is classical and well-posed (see for instance \cite{Ciarlet}, \cite{San89}). It is elliptic of total order 8, and the boundary conditions satisfy the SL condition. The system of equations is obtained by integration by parts, which yields:
\begin{equation}\label{pbvar4bis}
\left\{\begin{array}{l}
  - T^{\alpha \gamma }_{|\alpha}(u)+ \varepsilon ^2 b_{\beta }^ \gamma M^{\alpha \beta }_{|\alpha}(u)  +\varepsilon ^2\Big( b_{\alpha }^\gamma M^{\alpha \beta}(u)\Big)_{| \beta} =f^\gamma \\
 -b _{\alpha \beta }T^{\alpha \beta}(u)- \varepsilon ^2 M^{\alpha \beta } (u)_{| \alpha \beta}+ \varepsilon ^2b_{\alpha }^{\gamma } b_{\gamma \beta } M^{\alpha \beta } (u)=f^3,
\end{array}\right.
\end{equation}
where the flection moments $M^{\alpha \beta}$ were defined in (\ref{defdeB1}), (\ref{defdeB})  
The boundary conditions on $\Gamma _0$ (supposed clamped) are:
\begin{equation}\label{conditionlimiteprincipalebis}
u_1=u_2=u_3= \partial _n u_3=0 \textrm{ on }\Gamma _0
\end{equation}
while the  \textit{natural} boundary conditions on $\Gamma _1$ are in number of four, are not relevant (they are boundary terms obtained by integration by parts). We have:
\begin{proposition}\label{shapirolopatinskii2}
The boundary value problem  associated with the variational problem (\ref{pbvar}) when $\varepsilon >0$ considered as a system of three equations  with the unknowns $u$ 
is elliptic of total order $8$ with indices $(1,1,2)$ for the unknowns and the equations.
\end{proposition}

\section{A sensitive singular perturbation problem arising in the Koiter linear shell theory}\label{Sectionmodel}
\noindent

Very few is known concerning elliptic problems with boundary conditions not satisfying the SL condition and there is no general theory concerning them to our knowledge. Linear shell theory is one physical theory where they are naturally involved.

\subsection{Definition of the problem}\label{sectionmodel}
\noindent

Let us first recall the variational problem (\ref{pbvar}) we are interested in:
\begin{equation}\label{convergenceajout1}
\left\{\begin{array}{l}\textrm{Find } u^\varepsilon  \in V
\textrm{ such that, } \forall v \in V\\
a(u^\varepsilon ,v)+\varepsilon ^2 b(u^\varepsilon ,v)=\dualprod{f}{v},
\end{array}\right.
\end{equation}
where $f \in V'$ is given,  the brackets denote the duality between $V'$ and $V$.
More precisely, we consider the limit boundary partial differential system associated with (\ref{convergenceajout1}) when $\varepsilon =0$. This is the membrane system, which according to proposition \ref{shapirolopatinskii1}, is elliptic, satisfies the SL on $\Gamma _0$ but does not on $\Gamma _1$.

\subsection{Sensitive character}\label{sectionsensitive}
\noindent

Let us now recall the definition of sensitive problem. For a more complete description, see \cite{EgorovMeunierSanchez} and \cite{MeunierSanchez2}. Let us comment a little on proposition \ref{shapirolopatinskii1}.

The SL condition is not satisfied on a free boundary when $\varepsilon =0$ for the variational problem (\ref{convergenceajout1}). Specifically, the membrane problem is of total order four for elliptic surfaces. The number of boundary conditions should be two. On a fixed boundary $\Gamma _0$ they are:
\begin{equation}\label{sensitiveajout1}
u_1=u_2=0.
\end{equation}
Note that the trace of $u_3$ does not make sense in the membane framework. The previous boundary conditions satisfy the SL condition. Oppositely, on the free boundary $\Gamma _1$ the conditions are:
\begin{equation}\label{sensitiveajout2}
T^{\alpha \beta } (u)n_\beta =0.
\end{equation}
Let us admit that (\ref{pbvar3bispreambule}) has (in some sense) a solution $u$. Replacing it in the three equations (\ref{pbvar3bispreambule}) and in the boundary conditions on $\Gamma _1$ of (\ref{pbvar3bispreambule}), one obtains that the corresponding $T^{\alpha \beta }(u)$ satisfy the elliptic membrane tensions system with Cauchy conditions on the part of $\Gamma _1$ of the boundary. As this last problem has in general no solution in usual spaces, it follows that the membrane problem (\ref{pbvar3bispreambule}) cannot (in general) have solution in usual spaces. We shall see that existence of the solution (as well as the convergence for $\varepsilon \to 0$) only holds in very abstract spaces (out of the distribution space).

On the other hand, the boundary condition (\ref{sensitiveajout1}) constitutes the Cauchy condition for the rigidity system $\gamma_{\alpha \beta }(u)=0$. According to the uniqueness theorem for elliptic Cauchy problem (
see proof of Lemma \ref{rigide}) an elliptic shell is inhibited (or geometrically rigid) provided that it is fixed (or clamped) on a part (or the whole) of the boundary. When the boundary is everywhere free, the shell is not inhibited. Coming back to the inhibited elliptic shells, we see that when the whole boundary is fixed, the membrane problem is classical (the boundary condition satisfies the SL condition). But, when a part of the boundary $\Gamma _0$  is fixed whereas another one $\Gamma _1$ is not, the boundary conditions satisfy the SL condition on $\Gamma _0$ but not on $\Gamma _1$. This problem is out of the classical theory of elliptic boundary value problems and is called sensitive for reason which will be self evident later.

Let us consider formally the variationnal formulation of the membrane problem (\ref{pbvar3bispreambule}) (i.e. with $\varepsilon =0$):
\begin{equation}\label{sensitiveconvergenceajout1}
\left\{\begin{array}{l}\textrm{Find } u  \in V_a
\textrm{ such that, } \forall v \in V_a\\
a(u ,v)=\dualprod{f}{v},
\end{array}\right.
\end{equation}
where
$V_a$ is the completion of the "Koiter space" $V$  with the norm $ \| v\| _{a}=a(v,v)^{1/2}$.

The fact that $\|v\|_{a}$
is a norm on $V$ follows from lemma \ref{rigide}.

At the present state, it should be noticed that the previous completion process  is somewhat abstract
and the elements of $V_a$ are not necessarly distributions. Indeed, as the SL condition is not satisfied on $\Gamma _1$, we may construct corresponding solutions with $u \neq 0$ and $\gamma _{\alpha \beta}(u)=0$ which are rapidly oscillating along $\Gamma _1$ and exponentially decreasing inwards $\Omega$. This is only concerned with the higher order terms. When taking into account lower order terms  (which are "small" for rapidly oscillating solutions), we see that we may have "large $u$" with "small $\gamma_{\alpha \beta}(u)$" (i.e. small $\Sigma _{\alpha , \beta }\|\gamma_{\alpha \beta }(u) \|_{L^2}$) and then small membrane energy. Accordingly, the dual $V'_a$ where $f$ must be taken for (\ref{sensitiveconvergenceajout1}) to make sense is "very small".

The above property originates the term "sensitive". The problem is unstable. Very small  and smooth variations of $f$ (even in $\mathcal{D}(\Omega)$) induce modifications of the solution which are large and singular (out of the distribution space).

\subsection{Abstract convergence results as $\varepsilon \to 0$}\label{sectionsimple}
\noindent

In this section we recall some abstract convergence results  (in the norm of the specified spaces), see \cite{Caillerie} and \cite{EgorovMeunierSanchez} for more details.


Recalling the problem we are studying, we know that the shell is geometrically rigid:
\begin{equation}\label{rigiditeajout}
v \in V \textrm{ and } \gamma_{\alpha \beta }(v)=0 \Longrightarrow v =0.
\end{equation} 

Let $A$ and $B$ be the continuous operators from $V$ into $V'$ associated with the forms $a$ and $b$ by:
 \begin{equation}\label{2.13}
\dualprod{Au}{v}=a (u,v) \textrm{ and } 
\dualprod{Bu}{v}=b (u,v) \ \forall u,v \in V ,
\end{equation}
so that equation (\ref{convergenceajout1}) becomes:
\begin{equation}\label{convergenceajout2}
Au^\varepsilon +\varepsilon ^2 B u^\varepsilon =f.
\end{equation}

\begin{lemma}\label{lemme2.3}
The operator $A$ is injective and its range, $\mathcal{R}(A)$, is dense in $V'$.
\end{lemma}

The proof is not difficult, see \cite{EgorovMeunierSanchez} if necessary.

It then appears that the operator $A$ is a one-to-one mapping of $V$ onto $\mathcal{R}(A)$, which is a dense subset of $V'$. Let us define a new norm by
\begin{equation}\label{2.16}
\|v\|_{V_A}=\|Av\|_{V'} .
\end{equation}
Obviously $V$ is not complete for the previous norm. But $A$ defines an isomorphism between $V$ (with the norm $V_A$) and $\mathcal{R}(A)$ (with the norm $V'$). Automatically, $A$ has an extension by continuity which is an isomorphism between the completions of both spaces. Denoting by $\overline{A}$ the extended operator and by $V_A$ the completion of $V$ with the norm (\ref{2.16}), $\overline{A}$ is an isomorphism between $V_A$ and $V'$ (which is the completion of $\mathcal{R}(A)$ with the norm of $V'$). Equation (\ref{convergenceajout2}) may be written as well:
\begin{equation}\label{convergenceajout3}
\overline{A}u^\varepsilon +\varepsilon ^2 B u^\varepsilon =f.
\end{equation}

\begin{remark}
In order to pass to the limit as $\varepsilon \to 0$, the classical way consists in obtaining an a priori energy estimate of $u^{\varepsilon}$ by taking the duality product of  (\ref{convergenceajout3}) with $u^{\varepsilon}$. But such a way needs a hypothesis of boundedness of the functional $f$ with respect to the limit form $a$ and this does not work for any $f  \in V'$. In the general case, following an idea of Caillerie \cite{Caillerie}, see also \cite{EgorovMeunierSanchez}, which consists in proving that the term $\varepsilon ^2 B u^\varepsilon $ tends to zero in $V'$, one can pass this latter term  to the right-hand side, and show that it tends to $f$ in $V'$. Then using the fact that $\overline{A}$ is an isomorphism, it is possible to prove the existence of a limit of $u^\varepsilon $ in $V_A$. Specifically we have the following result. 
\end{remark}

\begin{theorem} \label{theorem2.6}
There exists a unique element $u^0$ in $V_A$  such that
\begin{equation}\label{A1}
\overline{A}u^0 =f.
\end{equation}
Moreover the following strong convergence holds in $V_A$:
\begin{equation}\label{convergenceajout4}
u^\varepsilon \to u^0 \textrm{ as } \varepsilon \to 0,
\end{equation}
where $u^\varepsilon \in V$ is the solution of  (\ref{convergenceajout3}).
\end{theorem}
The proof, which follows the trends outlined above, may be seen in \cite{EgorovMeunierSanchez}.

\begin{remark}
It should be emphasised that theorem \ref{theorem2.6} holds true without special hypothesis on $f$ (besides the obvious one $f \in V'$). The limit $u^0 \in V_A$ is the solution of the abstract problem (\ref{sensitiveconvergenceajout1}), which is not a variational one. The classical variational theory of the limit needs a supplementary hypothesis on $f$ : there exists $C>0$ such that 
\begin{equation}\label{classique}
\| \dualprod{f}{v} \| \le C a\scalprod{v}{v} ^{1/2}, \quad \forall v \in V,
\end{equation}
which is very restrictive in shell theory.
\end{remark}

For the sake of completness, let us give the elements of the classical limit theory under the assumption (\ref{classique}).

We first note that in such a case, $a\scalprod{v}{v} ^{1/2}$ defines a norm on $V$. Let $V_a$ be the completion of $V$ with respect to that norm (which should not be confused with $V_A$). We then note that (\ref{classique}) shows that $f$ may be extended by continuity to an element of $V_a'$. We shall denote this extension by $f$ again. Obviously, the variational problem 
\begin{equation}\label{variationalajout1}
\left\{\begin{array}{l}\textrm{Find } u^0  \in V_a
\textrm{ such that, } \forall v \in V_a\\
a(u^0 ,v)=\dualprod{f}{v},
\end{array}\right.
\end{equation}
is well posed and has a unique solution. We then have the classical convergence result (see \cite{Huet} e.g. or even \cite{San1})
\begin{theorem} \label{theorem2.7}
Under the assumption (\ref{classique}), we have 
\begin{equation}\label{convergenceajout6}
u^\varepsilon \to u^0 \textrm{ strongly in  } V_a \textrm{ as } \varepsilon \to 0,
\end{equation}
where $u^\varepsilon $ and $u^0$ are the solutions  of  (\ref{convergenceajout1}) and (\ref{A1}) respectively.
\end{theorem}


Let us now briefly recall  the non-inhibited case when (\ref{rigiditeajout}) does not hold.  In such a situation, there is a convergence result towards a limit with vanishing membrane energy. More precisely, we define the kernel $G$ of $a$:
\begin{equation}\label{defdeGajout}
G=\{ v \in V ; \ \gamma_{\alpha \beta }(v)=0\}=\{v \in V ; \ a(v,v)=0\}.
\end{equation} 
It is to be noticed that $G$ is a Hilbert space with the norm of $V$. But as $a(v,v)=0$ in $G$, we see that the norm of $V$ in $G$ is equivalent to $b(v,v)^{1/2}$. As a consequence,  the problem 
\begin{equation}\label{variationalajout10}
\left\{\begin{array}{l}\textrm{Find } v^0  \in G
\textrm{ such that, } \forall w \in G\\
b(v^0 ,w)=\dualprod{f}{w},
\end{array}\right.
\end{equation}
is well posed and has a unique solution. Moreover, since the "limit form" $a$ in (\ref{convergenceajout1}) vanishes on $G$, it implies some kind of weakness in $G$. The solution will be very large and we should define a new scaling in order to have a finite limit, $v^ \varepsilon = \varepsilon ^2 u^\varepsilon$ ,
 (\ref{convergenceajout1})  becomes 
\begin{equation}\label{variationalajout11}
\left\{\begin{array}{l}\textrm{Find } v^\varepsilon  \in V
\textrm{ such that, } \forall w \in V\\
\varepsilon ^{-2} a(v^\varepsilon ,w)+b(v^\varepsilon,w)=\dualprod{f}{w},
\end{array}\right.
\end{equation} 
we then have, see \cite{Nantes} e.g. for the proof \begin{theorem}
Under the assumption $G  \neq \emptyset$, 
\begin{equation}
v^\varepsilon \to v^0 \textrm{ strongly in } V,
\end{equation}
where $v^\varepsilon $ and $v^0$ are the solutions of (\ref{variationalajout11}) and (\ref{variationalajout10}) respectively.
\end{theorem}


\section{Heuristic asymptotics in the previous problem}\label{heuristics}
\noindent

The aim of this section is the construction, in a heuristic way,
of an approximate description of the solutions $u^{\varepsilon}$
of the linear Koiter model for small values of
$\varepsilon$. Indeed, coming back to the Koiter problem for $\varepsilon >0$, in the sensitive case, the problem is not really to describe the limit problem (which in general has no solution in the distribution space; in particular the space $V_A$ (see (\ref{2.16})) where there is always a limit, is not a distribution space), but rather to give a good description of the solution $u^\varepsilon $ for very small values of $\varepsilon$. This we shall try to do. We shall see that heuristic considerations allow to construct a simplified model accounting for the main features of the problem. 

To do so  we shall use the heuristic procedure of \cite{EgorovMeunierSanchez}. In this latter article,  we addressed a model problem including
a variational structure, somewhat analogous to the problem studied here,
but simpler, as concerning an equation instead of a system. It is
shown that the limit problem contains in particular an elliptic
Cauchy problem. This problem was handled in both a rigorous (very
abstract) framework and using a heuristic procedure for exhibiting
the structure of the solutions with very small $\varepsilon$. The main difference is that in the present work, we deal with systems instead of single equations.

We shall see that heuristic considerations involving minimization of energy allow us to reduce the problem to another on the boundary $\Gamma _1$. In that context, it is seen that the "pathological" operator $A$ is represented by a smoothing operator $S$ (i.e. sending any distribution to a $\mathcal{C}^\infty $ function), whereas the "classical" operator $B$ is represented by a "classical" elliptic operator $Q$. Denoting by $s(x,\xi)$ and $q(x,\xi)$ the corresponding symbols (here $x$ is the arc on $\Gamma _1$), $s$ is likely exponentially decreasing for $\xi \to  \infty$, whereas $q$ is algebraically growing. 
The action of $S+\varepsilon^{2}Q$ on test functions
is given by:
\begin{equation}\label{S.1}
(S+\varepsilon^{2}Q)\theta(x)=(2\pi)^{-1}\int_{-\infty}^{+\infty}e^{i
\xi x}[s(x,\xi)+\varepsilon^{2}q(x,\xi)]\tilde{\theta}(\xi)\D \xi .
\end{equation}

It is then apparent that, when $\varepsilon$ is small, operator
$S$ is significant only for bounded values of $\xi$, whereas
$\varepsilon^{2}Q $ describes the behavior for
$\xi\rightarrow\infty$. If $|\xi|<< log (1/\varepsilon) $, then the symbol of the operator $S+\varepsilon^{2}Q$ is equal to $(1+o(1))s(x,\xi)$  and for $|\xi|>>log (1/\varepsilon) $, it is $(1+o(1))\varepsilon ^2q(x,\xi)$. The balance of $S$ and $\varepsilon^{2}Q $
is obtained for values of $\xi$ such that:
\begin{equation}\label{S.2}
|\xi|\sim log (1/\varepsilon).
\end{equation}

This is the window of frequencies allowing a good description of
the simultaneous influence of $S$ and $\varepsilon^{2}Q $, which
is precisely our aim.  Moreover, it is easily seen that the range of frequencies (\ref{S.2}) is responsible for most of the contribution to the integral (\ref{S.1}). This property is of
great interest for the construction of the heuristic
approximation. More precisely, the heuristics incorporate
approximations for large $|\xi|$. This amounts to saying
that only the most singular parts of the solutions are retained,
or equivalently, that the approximate solutions are defined up to
more regular terms. This is for instance the kind of approximation
which is used in the construction of a parametrix. We also note
that, as (\ref{S.2}) involves "moderately large" values of
$|\xi|$, the "general quality" of the approximation is not
very good, as it is only accurate for very very small values of
$\varepsilon$.

 It should be noticed that numerical computations \cite{Bechet} carried out with very reliable software (including an adapted mesh procedure) for the Koiter problem with very small values of $\varepsilon $ agree with the overall trends of our heuristic procedure. It appears that most of the deformation consists in very large deformations along $\Gamma _1$ exponentially decreasing inwards $\Omega$ (then in good agreement with the "local lack of uniqueness" implied by the non-satisfied SL condition). As $\varepsilon $ decreases, the amplitude increases, whereas the wave length decreases very slowly, verifying fairly well (\ref{S.2}). The paper \cite{Bechet} also contains numerical comparisons with the case when the shell is fixed all along its boundary, which is classical (as the SL condition is satisfied all along the boundary). The differences are drastic for small values of $\varepsilon$.

\subsection{Introduction to the heuristic asymptotic}
\noindent

A first remark in the context described above is that sensitive problems may be considered as "intermediate" between "inhibited"  and "non-inhibited". Indeed, "inhibited" means that $v \in V$ and $\gamma _{\alpha \beta}(v)=0$ implies $v=0$, whereas "non-inhibited" means that there are non vanishing elements $v$ of $V$ such that $\gamma _{\alpha \beta }(v)=0$. Strictly speaking, sensitive problems enter in the class "inhibited", but there are non vanishing elements $v$ of $V$ with "very small" $\gamma_{\alpha \beta}(v)$.

In order to minimize the energy
\begin{equation}\label{energiesensitive}
a(v,v)+ \varepsilon ^2 b(v,v) -2 \dualprod{f}{v},
\end{equation}   
it is clear that we may proceed as in non-inhibited problems. The solution with small $\varepsilon $ "avoids" the (larger) membrane energy $a$, so that roughly speaking, solutions for small $\varepsilon$ should have
$\gamma _{\alpha \beta }(v)$ vanishing or at least very small with respect to $v$. 

Obviously, it is impossible to impose the four boundary conditions (\ref{conditionlimiteprincipalebis}) on $\Gamma _0$ with the "exact" system $\gamma _{\alpha \beta }(v)=0$ as they imply $v=0$.

Nevertheless, we shall see in Section \ref{6.2.1} that it is possible to construct functions satisfying  the two boundary  conditions $u_n=u_t=0$ on $\Gamma _0$ with the "non exact" system $\gamma _{\alpha \beta }(v)=0$ in the sense that $\gamma _{\alpha \beta }(v)$  will be "very small" (i.e.  $\Sigma _{\alpha , \beta }\|\gamma_{\alpha \beta }(v) \|_{L^2}$ will be very small). This will imply a "membrane boundary layer" in the vicinity of $\Gamma _0$ involving the bilinear form $a$. To this end, we shall first construct a set of functions $v$ with only one vanishing component on $\Gamma _0$. Choosing (for instance) the normal component, we define:

\begin{equation}\label{definitiondeg0}
G^0=\{v, \ \gamma _{\alpha \beta }(v)=0 \textrm{ on } \Omega , \ v_2=0 \textrm{ on } \Gamma _0\},
\end{equation}
the regularity is not precised as we shall later take the completion, we may consider $C^{\infty }$ functions for instance. It is to be noticed that $v$ is a triplet of functions.

Recalling Remark \ref{resultatutile}, we know that 
up to a  finite dimensional space composed of smooth functions,  the space
$G^0$ is isomorphic to the space of traces on $\Gamma _1$:
\begin{equation}
\{w \in C^{\infty }(\Gamma _1)\}
\end{equation}
the isomorphism is obtained by solving the problem:
\begin{equation}\label{pbdirichletajout12}
\left\{\begin{array}{l}\gamma_{\alpha \beta } (\tilde{w}) =0\textrm{ on }  \Omega,\\
\tilde{w}_2 =0\textrm{ on }  \Gamma _0,\\
\tilde{w}_3 =w\textrm{ on }  \Gamma _1.
\end{array}\right.
\end{equation}

In the sequel, when we will consider a function $\tilde{w} \in G^0$, we will consider a function of the equivalence class for the quotient operation described in Remark \ref{resultatutile}. Moreover, we shall consider indifferently the functions $\tilde{w} $ obtained after a quotient operation on $\overline{\Omega }$ (for the finite dimensional space) or their traces $w$ on
$\Gamma _1$.

Moreover, the conditions $u_3=\partial _n u_3=0$ on $\Gamma _0$ of 
(\ref{conditionlimiteprincipalebis}) will be satisfied with the help of a "flection sublayer" involving the bilinear form $b$; its effect is not relevant (see Section \ref{6.2.2}).

According to the previous considerations, we shall consider the minimization problem on $G^0$ instead of on $V$. This modified problem obviously involves the $a$-energy and the
$\varepsilon^{2}b$-energy. A natural space for handling it should
be the completion $G$ of  $G^0$ with the corresponding norm.

The fact that we may "neglect" the functions in the finite dimension space of smooth functions follows from the fact that we are interested in the singular part.

\subsection{The boundary layer on $\Gamma _0$}
\noindent

Let $\tilde{w}$ be in $G^0$ (see (\ref{definitiondeg0})) and let $\varepsilon >0$ be fixed. The aim of this section is to build a modified 
function $\tilde{w}^a$ of $\tilde{w}$ in  a narrow boundary layer of $\Gamma _0$ in order to satisfy the
 supplementary boundary conditions $ \tilde{w}_t=\tilde{w}_3=\partial _n \tilde{w}_3=0$ on $\Gamma _0$.

The present problem is analogous to the "model problem" of  \cite{EgorovMeunierSanchez} in the case of a singular perturbation, i.e. \cite{EgorovMeunierSanchez} Section 7.1.2. Indeed, the membrane problem is of total order 4 allowing 2 boundary conditions ($\tilde{w}_t=\tilde{w}_n=0$) on $\Gamma _0$, whereas the complete Koiter shell problem is of order 8, allowing 4 boundary conditions (we shall add $\tilde{w}_3=\partial _n \tilde{w}_3=0$) on $\Gamma _0$. It appears that the two first conditions ($\tilde{w}_t= \tilde{w}_n=0$) may be obtained from elements of $G^0$ by modifying them on account of a "membrane layer" which relies on the membrane system, of thickness of order $\frac{1}{\log(1/\varepsilon)}$ on $\Gamma _0$, whereas an irrelevant boundary layer will be considered in Section \ref{6.2.2}.

\subsubsection{The membrane boundary layer on $\Gamma_0$}\label{6.2.1}
\noindent

In this subsection, we proceed to modify the element $\tilde{w}$ of $G^0$ in order to satisfy both conditions $u_1=u_2=0$ on $\Gamma _0$.

Let $\tilde{\Gamma }_0$ be a neighborhood of $\Gamma _0$ in
$\R^{2}$ disjoint with $\Gamma _1$ and sufficiently narrow to be
described by the curvilinear coordinates $y_{1}=$ arc of
$\Gamma_{0}$ and $y_{2}=$ distance along the normal to
$\Gamma_{0}$. Let  $(\psi _j (y_1))_{j \in J}$ be a partition of
the unity associated with $\Gamma _0$ and let $\eta \in C^{\infty
}(\R_{+} ;\R_{+})$ be a cut-off function equal to $1$ for small values
of $y_{2}$.

 The mappings
$\theta _j$  defined by $\theta _j(y_1,y_2)=\psi _j (y_1)\eta (y_2)$, where  $y_2$ is the (inwards) normal
coordinate along $\Gamma _0$, define a partition of unity in  $\tilde{\Gamma }_0 $; in particular, for a given $\tilde{w} \in G^0$, we have:
\begin{equation}\label{partition1}
\forall (y_1, y_2) \in \tilde{\Gamma }_0, \ \tilde{w}(y_1,y_2)=\Sigma _{j \in J} \theta _j(y_1, y_2) \tilde{w}(y_1,y_2).
\end{equation}

Let us now fix $j $ in $J$ and $y_2$  such that $(y_1,y_2) \in \tilde{\Gamma }_0$, the
function $\theta _j(\cdot, y_2) \tilde{w}(\cdot,y_2)$ has a compact support, we denote
by $\tilde{w}^j(\cdot , y_2)$ its extension by zero to $\R$ and by $\mathcal{F} (\tilde{w}^j)$ the tangential
  Fourier transform, $y_1 \to \xi _1$, of $\tilde{w}^j$.

Let us first exhibit the local structure of the Fourier transform
of $\tilde{w}^{j}$ close to $\Gamma_{0}$. Denoting by $\theta
_j$ the multiplication operator by $\theta _j$, recalling that the
commutator of the operator $\gamma $ associated with $\gamma_{\alpha \beta} $  and $\theta _j$, denoted by
$[\gamma,\theta _j] $, is a differential operator of lower order, taking the $\gamma $ operator in the new coordinates
 $(y_1,y_2)$ (which, according to our approximation close to $\Gamma_{0}$,
  has the same principal part) and using  that $\tilde{w}\in G^0$,  we see that:
\begin{equation}\label{approx}
 \gamma_{\alpha \beta}( \tilde{w}^j ) +
  U_{\alpha \beta}(y,D)   \tilde{w}^j =0 \textrm{ on } \R \times (0,t),
\end{equation}
for some $t>0$, $U_{\alpha \beta}$ being differential operator of order less
than the order of $ \gamma_{\alpha \beta}$. 

Now, according to the general trends of our boundary layer approximation, we can
neglect the terms of lower order in (\ref{approx}) and we can
proceed as in the construction of a parametrix (freezing
coefficients, dropping lower order terms, solving such simpler
equation via  tangent Fourier transform and gluing together the
solutions for different $j$), so that (\ref{approx}) becomes
\begin{equation}\label{approx2}
 \gamma_{\alpha \beta}( \tilde{w}^j )  =0 \textrm{ on } \R \times (0,t).
\end{equation}

The previous system can be rewritten as
\begin{equation}\label{approx2bis}
\left\{\begin{array}{l}
\frac{\partial }{\partial y_1}\tilde{w}^j_{1}-b_{11} \tilde{w}^j_{3} =0 ,\\
\frac{\partial }{\partial y_2}\tilde{w}^j_{2}-b_{22} \tilde{w}^j_{3} =0 , \\
\frac{1}{2}\Big(\frac{\partial }{\partial y_2}\tilde{w}^j_{1}+\frac{\partial }{\partial y_1}\tilde{w}^j_{2}\Big)-b_{12} \tilde{w}^j_{3}=0 ,
\end{array}\right.
\end{equation}
and taking the tangential Fourier transform denoted by $\mathcal{F}(\tilde{w}^j) (\xi _1, y_2)$ this yields   
\begin{equation}\label{approx4}
\Big(\hat{\gamma}_0+\tilde{\gamma}_1\frac{\D }{\D y_2}\Big) \mathcal{F}(\tilde{w}^j)=0,
\end{equation}
with
$$\hat{\gamma }_0 =\left(\begin{array}{lll}
-i \xi _1 & 0 &- b_{11} \\
0 & 0 & -b_{22}\\
0 & -i \xi _1 &- 2 b_{12}
\end{array}\right)\textrm{ and }\tilde{\gamma }_1 =\left(\begin{array}{lll}
0 & 0 &0 \\
0 & 1 & 0\\
1 & 0&0
\end{array}\right).$$

The general solution of the system (\ref{approx4}) is:
\begin{equation}\label{solgenerale2}
\mathcal{F}(\tilde{w}^j)(\xi_1,y_2)=A
\tilde{w}_+e^{\lambda_+(\xi _1)y_2}+B
\tilde{w}_-e^{\lambda_-(\xi _1)y_2},
\end{equation}
with $\tilde{w}_+ =\left(\begin{array}{l} i\frac{\lambda_+(\xi _1)}{\xi_1}\frac{b_{11}}{b_{22}}
\\
1\\  
\frac{\lambda_+(\xi _1)}{b_{22}}
\end{array}\right)
$, $\tilde{w}_-=\left(\begin{array}{l} 
i\frac{\lambda_-(\xi _1)}{\xi_1}\frac{b_{11}}{b_{22}}
\\
1\\  
\frac{\lambda_-(\xi _1)}{b_{22}}
\end{array}\right)
$, $\lambda _+(\xi _1)= -i\xi _1 \frac{b_{12}}{b_{11}}+\frac{|\xi _1|}{
b_{11}}\sqrt{b_{11}b_{22}-b_{12}^2}$ and $  \lambda _- (\xi _1)= -i\xi _1 \frac{b_{12}}{b_{11}}-\frac{|\xi _1|}{
b_{11}}\sqrt{b_{11}b_{22}-b_{12}^2}$. 

Since $\tilde{w} \in G_0$, it follows that  $\mathcal{F}(\tilde{w}_{2}^j) (\xi _1, 0)=0$ 
hence $A=-B$. Consequently, we deduce that
\begin{equation}\label{solmodifie2}
\mathcal{F}(\tilde{w}^j) (\xi _1, y_2)=\frac{
b_{11}b_{22}}{2 |\xi _1| \sqrt{b_{11}b_{22}-b_{12}^2}} \mathcal{F}(\tilde{w}^j_{3}) (\xi _1, 0)\Big(\tilde{w}_+e^{\lambda_+(\xi _1)y_2}-\tilde{w}_-e^{\lambda_-(\xi _1)y_2}\Big).
\end{equation}

This expression exhibits the structure of (the Fourier transform of) $\tilde{w}$ in a narrow neigbourhood of $\Gamma _0$. It was expressed in terms of (the Fourier transform of) the trace of its third component on $\Gamma _0$, but this choice is arbitrary.

We now proceed to the modification of $ \tilde{w}^j$ in
$\tilde{w}^{j a}$ in a narrow boundary layer of $\Gamma _0$ in
order to satisfy (always within our approximation) the equation
coming from (\ref{pbvar4bis}) for $\varepsilon =0$ (this is the membrane boundary layer associated with the membrane system of Section \ref{membranesystem}). Using
considerations similar to those leading to (\ref{approx}), 
this
amounts to
\begin{equation}\label{eqmodifie1}
\Big(\tilde{\gamma } ^*\tilde{A}_1 \tilde{\gamma } \Big)\tilde{w}^{ja} +
U(y,D)   \tilde{w}^{ja} = 0 \textrm{ on } \R \times (0,t),
\end{equation}
where $U$ is a differential operator of lower order than four, $\tilde{\gamma }^*$ denotes the operator:
$$\tilde{\gamma }^* =\left(\begin{array}{lll}
\partial _1 & 0 &- b_{11} \\
0 & \partial _2 & -b_{22}\\
\partial _2 & \partial _1 &- 2 b_{12}
\end{array}\right),$$
and 
$$\tilde{A}_1=\left(\begin{array}{lll}
A^{1111} & A^{1122} &A^{1112} \\
A^{2211} & A^{2222} & A^{2212}\\
A^{1211} & A^{1222} &A^{1212}
\end{array}\right).$$

Therefore dropping as before terms of lower order, we have:
\begin{equation}\label{eqmodifie10}
\Big(\tilde{\gamma } ^*\tilde{A}_1 \tilde{\gamma } \Big)\tilde{w}^{ja} = 0 \textrm{ on } \R \times (0,t),
\end{equation}
which can be rewritten as 
\begin{equation}\label{eqmodifie10bis}
\Big((\tilde{\gamma } _0  ^*- \tilde{\gamma }_1^*
\partial_2)\tilde{A}_1 (\tilde{\gamma } _0  + \tilde{\gamma }_1
\partial _2) \Big)\tilde{w}^{ja} = 0 \textrm{ on } \R \times (0,t),
\end{equation}
with $\tilde{\gamma }^*= \overline{\tilde{\gamma }}^T$ and
$$\tilde{\gamma }_0 =\left(\begin{array}{lll}
\partial _1 & 0 &- b_{11} \\
0 & 0 & -b_{22}\\
0 & \partial _1 &- 2 b_{12}
\end{array}\right).$$
Hence taking the tangential Fourier transform, we look for solutions of the system:
\begin{equation}\label{eqmodifie100bis}
\Big((\overline{\hat{\gamma } _0 } ^T- \tilde{\gamma }_1^T\frac{\D }{\D y_2})\tilde{A}_1 (\hat{\gamma } _0  + \tilde{\gamma }_1\frac{\D }{\D y_2}) \Big)\mathcal{F}\Big(\tilde{w} ^{ja}\Big)(\xi _1, y_2)= 0 ,
\end{equation}
with 
$$\hat{\gamma }_0 =\left(\begin{array}{lll}
-i \xi _1 & 0 &- b_{11} \\
0 & 0 & -b_{22}\\
0 & -i \xi _1 &- 2 b_{12}
\end{array}\right).$$

At this moment, it is worthwhile to compare  (\ref{eqmodifie100bis}) and (\ref{approx4}). We see that the given function $\tilde{w}^j$ (rather its Fourier transform) solves the "right half" of  (\ref{eqmodifie100bis}), i.e. the expression on the right of $\tilde{A}_1$ in  (\ref{eqmodifie100bis}). Obviously, the "left half" accounts for the "adjoint part", coming with integration by parts from the bilinear form $a$ (see (\ref{defedeaajout})). Our aim in constructing the modified $\tilde{w}^{ja}$ is to satisfy the conditions $\tilde{w}^{ja}_1=\tilde{w}^{ja}_2=0$ on $y_2=0$, whereas for "large $y_2$" (in the sense of "out of the layer") the modified $\tilde{w}^{ja}$ coincides (up to small terms) with the given $\tilde{w}^{j}$. We now proceed to write down the general solution of (\ref{eqmodifie100bis}) on account of its special structure.

For $\lambda \in \{\lambda _-(\xi _1), \lambda _+(\xi _2)\}$, let us consider the function $k$ defined by: 
\begin{equation}\label{eqmodifie100terbis}
k(\xi _1, y_2) = (y_2 w +v)e^{\lambda y_2} ,
\end{equation}
where $w \in \{\tilde{w} _-, \tilde{w} _+\}$  is a solution of 
$$\Big(\hat{\gamma }_0 +\lambda \tilde{\gamma }_1\Big) w =0,$$
and $v$ is unknown. We then search for solutions of (\ref{eqmodifie100bis}) under the form (\ref{eqmodifie100terbis}) i.e.:
\begin{equation}\label{eqmodifie100ter}
\Big((\overline{\hat{\gamma } _0 } ^T- \tilde{\gamma }_1^T\frac{\D }{\D y_2})\tilde{A}_1 (\hat{\gamma } _0  + \tilde{\gamma }_1\frac{\D }{\D y_2}) \Big)k(\xi _1, y_2)= 0 ,
\end{equation}
We check that 
\begin{eqnarray*}
\Big(\hat{\gamma }_0 + \tilde{\gamma }_1\frac{\D }{\D y_2}\Big)\Big(y_2 w +e^{\lambda y_2} v\Big)=\Big((\hat{\gamma }_0+\lambda \tilde{\gamma }_1)v+\tilde{\gamma }_1  w\Big)e^{\lambda y_2}.
\end{eqnarray*}
So that (\ref{eqmodifie100ter}) becomes
\begin{eqnarray*}
\Big((\overline{\hat{\gamma } _0 } ^T- \tilde{\gamma }_1^T\frac{\D }{\D y_2})\tilde{A}_1 (\hat{\gamma } _0  + \tilde{\gamma }_1\frac{\D }{\D y_2}) \Big)\Big(y_2 w+ v\Big)e^{\lambda y_2}&=&\\
(\overline{\hat{\gamma } _0}  ^T- \tilde{\gamma }_1^T\frac{\D }{\D y_2})\tilde{A}_1  \Big((\hat{\gamma }_0+\lambda \tilde{\gamma }_1)v +\tilde{\gamma }_1  w \Big)e^{\lambda y_2}&=&0 .
\end{eqnarray*}
This amounts to saying that $\tilde{A}_1\Big((\hat{\gamma }_0+\lambda \tilde{\gamma }_1)v+\tilde{\gamma }_1  w  \Big)$ is an eigenvector of $\overline{\hat{\gamma } _0  }^T- \lambda \tilde{\gamma }_1^T$ associated with the zero eigenvalue. Since $\textrm{dim Ker }\Big( \overline{\hat{\gamma } _0}  ^T- \lambda \tilde{\gamma }_1^T\Big)=1$,  denoting by $u_0 $ a non vanishing vector of $ \textrm{Ker }\Big( \overline{\hat{\gamma } _0}  ^T- \lambda \tilde{\gamma }_1^T\Big)$,
then
$v$ should satisfy 
\begin{equation}\label{constructiondev}
(\hat{\gamma }_0+\lambda \tilde{\gamma }_1)v+\tilde{\gamma }_1  w    = \tilde{A}_1 ^{-1} (\tau u_0), \textrm{ for some } \tau \in \C.
\end{equation}
According to the Fredholm alternative, a necessary and sufficient condition for the existence of such a $v$ is that 
$$ \tilde{A}_1 ^{-1}(\tau u_0 ) - \tilde{\gamma }_1  w   \ \in (\textrm{Vect }{u_0})^{\perp}.$$
Since $\tilde{A}_1$ is positive definite, we deduce that  $ \scalprod{ \tilde{A}_1 ^{-1} u_0   }{u_0}>0$ hence 
$\tau = \frac{\scalprod{\tilde{\gamma}_1u}{u_0}}{\scalprod{\tilde{A}_1^{-1}u_0}{u_0}}$ satisfies
$$ \scalprod{\tau \tilde{A}_1 ^{-1} u_0  - \tilde{\gamma }_1 w 
 }{u_0}=0.$$

It follows that the vector $v \in \C ^3$ exists and is unique (up to an additive and arbitrary eigenvector, which is irrelevant in the sequel). Consequently, $k$ defined as above satisfies  (\ref{eqmodifie100ter}).

Repeating this argument twice (first for  $\lambda_+(\xi _1)$, and then for  $\lambda _- (\xi _1)$), and denoting by $v_+$ and $v_-$ the corresponding vectors $v$, we see that   
\begin{eqnarray}\label{solmodifie}
\mathcal{F}\Big(\tilde{w} ^{ja}\Big)(\xi _1, y_2) &=& C_1\tilde{w} _e^{-\lambda _+(\xi _1)y_2}+C_2 \tilde{w}_-e^{-\lambda _-(\xi _1)y_2}+C_3 \Big(y_2\tilde{w}_+ +v_+\Big)e^{\lambda _+(\xi _1)y_2} \nonumber \\
& &+C_4\Big(y_2 \tilde{w}_- + v_-\Big)e^{\lambda _-(\xi _1)y_2},
\end{eqnarray}
with arbitrary $C_1,C_2,C_3, C_4$ is the general solution of (\ref{eqmodifie100bis}).

We are now determining $C_1,C_2,C_3, C_4$ in order to satisfy the boundary conditions $\tilde{w}^{ja}_1=\partial _2\tilde{w}^{ja}_1=0$ at $y_2=0$ and the "matching condition" with $\tilde{w}^{j}$, i.e. in the context of boundary layer theory (for large $|\xi _1 |$), $\tilde{w}^{ja}$ should become $\tilde{w}^{j}$ out of the layer.

Let us now explain the process of matching the layer: out of the
layer, we want $\tilde{ w} ^{ja }$ to match with the given function
$ \tilde{ w}^j$. Since $|\xi _1 |>>1$, then $|\xi _1 | y_2 >>1$ and $ \frac{ \sqrt{b_{11}b_{22}-b_{12}^2}}{b_{11}} |\xi _1 |y_2>>1$ which
means that $y_2 >>\frac{b_{11}}{\sqrt{b_{11}b_{22}-b_{12}^2}}\frac{1}{|\xi _1|}$ (but we still impose that
$y_2$ is small in order to be in a narrow layer of $\Gamma _0$ where (\ref{solmodifie2}) holds);
this is perfectly consistent, as we will only use the functions
for large $|\xi_{1}|$, hence the terms with coefficients $ C_2$ and $C_4 $
 are "boundary layer terms" going to zero out of the
layer (i.e. for $|y_2 |>>\mathcal{O}\Big( \frac{1}{|\xi
_1|}\Big)$).

The matching with (\ref{solmodifie2}) out of the
layer then gives
\begin{equation}\label{transmission}
C_3=0 \textrm{ and } C_1 =
\frac{
b_{11}b_{22}}{2 |\xi _1| \sqrt{b_{11}b_{22}-b_{12}^2}} \mathcal{F}(\tilde{w}^j_{3}) (\xi _1, 0).
\end{equation}
The two other constants $C_2$ and $C_4$ are determined by
$$\mathcal{F}(\tilde{w}^{ja}  )_1 (\xi _1, 0)=0 \textrm{ and } \mathcal{F}(\tilde{w}^{ja}  )_2 (\xi _1, 0)=0,$$
which yields the existence of two constants $\alpha$ and $\beta$ such that
$$C_2=\alpha C_1\textrm{ and } C_4=\beta C_1.$$
So that the modified solution is of the form:
\begin{eqnarray} \label{solmodifie22bis}
\mathcal{F}(\tilde{w}^{ja} ) (\xi _1, y_2)&=& \frac{
b_{11}b_{22}}{2 |\xi _1| \sqrt{b_{11}b_{22}-b_{12}^2}} 
\Big(\tilde{w} _+e^{\lambda _+(\xi _1)y_2}\\
& &+((\alpha +\beta y_2 )\tilde{w}_- +\beta v_-)e^{\lambda _-(\xi _1)y_2}\Big)\mathcal{F}(\tilde{w}^j_{3}) (\xi _1, 0).\nonumber
\end{eqnarray}

The modification of the function
$\tilde{w}_{j}$ then consists in adding to it the inverse Fourier
transform of
\begin{equation}\label{V1}
 \frac{
 b_{11}b_{22}}{2 |\xi _1| \sqrt{b_{11}b_{22}-b_{12}^2}} \Big((\alpha+1 +\beta y_2 )\tilde{w}_- +\beta v_-\Big)e^{\lambda _-(\xi _1)y_2}\mathcal{F}(\tilde{w}^j_{3}) (\xi _1, 0).
\end{equation}

We shall study in the sequel the behavior of such an expression. The role of the constants $\alpha$ and $\beta$ is not relevant, and we may assume, for instance that $\alpha=-1$ and $\beta =1$ (this amounts to change $\tilde{w}_-$ and $\tilde{v}_-$). As the result, the modification of the function
$\tilde{w}_{j}$ consists in adding to it the inverse Fourier
transform of
\begin{equation}\label{V1bis}
 \frac{
 b_{11}b_{22}}{2 |\xi _1| \sqrt{b_{11}b_{22}-b_{12}^2}} \Big( y_2 \tilde{w}_- + v_-\Big)e^{\lambda _-(\xi _1)y_2}\mathcal{F}(\tilde{w}^j_{3}) (\xi _1, 0).
\end{equation}

More precisely, on account of considerations at the beginning of Section 6 (see in particular (\ref{S.1}) and (\ref{S.2})), the modification should only be effective for large $|\xi_{1}|$, accounting for "singular parts" of the solution. Moreover, in order to have $\tilde{w}^a \in V$, we shall also impose $\tilde{w}^{ja} _1= \partial _2\tilde{w}^{ja} _1=0$ on $\Gamma _0$ (the other two conditions $\tilde{w}_3^a = \partial _n \tilde{w}^a _3=0$ on $\Gamma _0$ will be adressed in Section \ref{6.2.2}). To this end, we multiply the added term by a cutoff function  avoiding low
frequencies (It should be remembered that this is one of the
typical devices in the construction of a parametrix).  More
precisely, on account of (\ref{S.2}), we shall only keep
frequencies of order more or equal than
$[log(1/\varepsilon)]^{1/2}$, which preserve the useful region
(\ref{S.2}) and are large (then consistent with the fact that the
modification is a layer).  Moreover, in order to the modified function satisfy the boundary conditions, we must also take into account the low frequencies of the Fourier transform which we multiply by a smooth vector $\rho (y_2)$ such that $\rho _1 (0)=\rho _2 (0)=0$ and $\rho (y_2)=0$ for $y_2>C$ for a certain $C$. The division into high and low frequencies is defined by a smooth function $H(z)$ equal to $1$ for $|z| > 1$ and vanishing for $|z|< 1/2$, with $z=\frac{\xi}{[log(1/\varepsilon)]^{1/2}}$. Finally, we define the function 
\begin{eqnarray}\label{V2}
& &h(\varepsilon, \xi, y_2) =(1-H(\frac{\xi_1}{[log(1/\varepsilon)]^{1/2}}))\rho (y_2)+ \\
 & & 
  \frac{
  b_{11}b_{22}}{2 |\xi _1| \sqrt{b_{11}b_{22}-b_{12}^2}} \Big( y_2\tilde{w}_- + v_-\Big)e^{\lambda _-(\xi _1)y_2}H(\frac{\xi}{[log(1/\varepsilon)]^{1/2}}),\nonumber
\end{eqnarray}
which obviously has its first and second components vanishing for $y_2=0$. 
Now we can modify the function $
\tilde{w}_{j}$ by
\begin{equation}\label{V3}
\delta \tilde{w}_{j}\equiv \tilde{w}^{a}_{j}-\tilde{w}_{j},
\end{equation}
where $\delta \tilde{w}_{j}$ is defined by its Fourier transform:
\begin{equation}\label{V4}
 \mathcal{F}\Big(\delta \tilde{w}_{j}\Big) =
\mathcal{F}(\tilde{w}^j_{3}) (\xi _1, 0)h(\varepsilon,\xi , y_2).
\end{equation}

\begin{remark}
The constant $C$ in the definition of $\rho (y_2)$ is chosen sufficiently small for this function to vanish out of the layer of $\Omega $ close to $\Gamma _0$ where the curvilinear coordinates $y_1, y_2$ operate. Rigorously speaking, the rest of the expression should also be multiplied by a cut-off function vanishing  for $y_2 >C$, but this is practically not necessary, as this part is exponentially  small for large $|\xi _1 |$. 
\end{remark}

Hence summing over $j$ and defining on $\Gamma_{0}$ the family
(with parameter $y_{2}$) of pseudo-differential smoothing
operators $\delta\sigma(\varepsilon,D_{1},y_{2})$ with symbol:
\begin{equation}\label{V5}
 \delta \sigma (\varepsilon,\xi _1, y_2)=  \frac{|b_{11}|b_{22}}{2 |\xi _1| \sqrt{b_{11}b_{22}-b_{12}^2}} \Big( y_2 \tilde{w}_- + v_-\Big)e^{\lambda _-(\xi _1)y_2}, 
\end{equation}
we see that the modification of the function $ \tilde{w}$:
\begin{equation}\label{V6}
\delta \tilde{w} =  \tilde{w}^{a}- \tilde{w}
\end{equation}
is precisely the action of $\delta\sigma(\varepsilon,D_{1},y_{2})$
on $\tilde{w}^j_{3} (y _1, 0)$.

Once $\tilde{w}^a$ is constructed, it is worthwhile computing its $a$-energy. This we proceed to do. More generally, we shall compute the form $a$ for two functions $\tilde{v}^a$ and $\tilde{w}^a$.
 
Let us now compute the leading terms of the $a$-energy  of the
modified function $\tilde{w}^a$.

Let $\tilde{v}$ and $\tilde{w}$ be two elements in $G^{0}$ and
$\tilde{v}^{a}, \tilde{w}^{a}$ the corresponding elements modified
in the boundary layer. As the given $\tilde{v}$ and $\tilde{w}$
satisfy $\gamma _{\alpha \beta }(\tilde{v})= \gamma _{\alpha \beta }(\tilde{w})=0$, the $a$-form is only concerned with the
modification terms $\delta\tilde{v}$ and $\delta\tilde{w}$. Then,
within our approximation, we have:
\begin{equation}\label{V8}
a( \tilde{v} ^a, \tilde{w}^a)=  \int _{\Gamma _0 } A^{\alpha \beta \lambda \mu }dy_1 \int _0
^{+\infty } \gamma_{\alpha \beta} (\delta \tilde{v})\overline{\gamma _{\lambda \mu}
(\delta \tilde{w}})\D y_2.
\end{equation}
where the integral in $\D y_{2}$ is only effective in the narrow
layer. Using the partition of the unity $\theta_{j}$ and denoting
as before by $\delta w_{j}(\cdot,y_{2})$ the extension with value
$0$ to $\mathbb{R}$ of $\theta_{j}(\cdot,y_{2})  \delta
w(\cdot,y_{2}) $, we have
\begin{equation}\label{V9}
a( \tilde{v} ^a, \tilde{w}^a)= \Sigma _{j, k} \int _{\Gamma _0 }A^{\alpha \beta \lambda \mu } d
y_1 \int _0 ^{+\infty } \gamma _{\alpha \beta } (\delta
\tilde{v}_{j})\overline{\gamma _{\lambda \mu}  (\delta \tilde{w}_{k}})\D y_2.
\end{equation}

Consequently, using the tangential Fourier transform $y_1 \to
\xi_{1} $ and the Parceval-Plancherel theorem, dropping lower order terms  (within our approximation, we only consider expressions with large $|\xi _1 |$ which amounts to take $H=1$ in (\ref{V2})), we deduce
that
\begin{eqnarray*}
& &a(\tilde{v}^a, \tilde{w}^a)= \\
 & &  \Sigma _{j,k} \int _{-\infty } ^{+
\infty}\tilde{A}_1 \D \xi _1 \int _0 ^{+\infty } \Big(\hat{\gamma _0}+\tilde{\gamma }_1\frac{\D }{\D y_2}\Big)  \delta \sigma (\varepsilon,\xi, y_2)\mathcal{F}(\tilde{v}^j_{3}) (\xi _1, 0)\big)\times \nonumber \\
& &\overline{\Big(\hat{\gamma _0}+\tilde{\gamma }_1\frac{\D }{\D y_2}\Big)  \delta
\sigma (\varepsilon,\xi, y_2) \mathcal{F}(\tilde{w}^k_{3}) (\xi _1, 0)\big)}\D y_2 =
 \\
& &\Sigma _{j,k} \int _{-\infty } ^{+
\infty} \tilde{A}_1\D \xi _1 \int _0 ^{+\infty }\frac{
b_{11}b_{22}}{2 |\xi _1| \sqrt{b_{11}b_{22}-b_{12}^2}} \Big((\hat{\gamma }_0+\lambda _- \tilde{\gamma }_1)v_-+\tilde{\gamma }_1  \tilde{w} _- \Big) e^{\lambda _- y_2}\mathcal{F}(\tilde{v}^j_{3}) (\xi _1, 0)\big)\times \nonumber \\
& &\overline{\frac{
b_{11}b_{22}}{2 |\xi _1| \sqrt{b_{11}b_{22}-b_{12}^2}} \Big((\hat{\gamma }_0+\lambda _-\tilde{\gamma }_1)v_-+\tilde{\gamma }_1  \tilde{w }_- \Big) e^{\lambda _- y_2} \mathcal{F}(\tilde{w}^k_{3}) (\xi _1, 0)\big)}\D y_2
\end{eqnarray*}
Hence, on account of the definitions of $\hat{\gamma }_0$, $\tilde{\gamma }_1$, $\lambda _-$ and $\tilde{w}_-$ 
integrating in $y_2$, we know that
\begin{equation}\label{energiea1}
a(\tilde{v}^a, \tilde{w}^a)=\Sigma _{j, k}\int _{-\infty } ^{+
\infty} \theta   |\xi _1|\mathcal{F}(\tilde{v}^{j})_{3|y_2=0} \overline{\mathcal{F}( \tilde{w}^{k})_{3|y_2=0}} h^{2}(\varepsilon,\xi , y_2) \D \xi _1,
\end{equation}
with $\theta = \theta (A^{\alpha \beta \lambda \mu}, (v_-)_1(0), b_{\alpha \beta}, \mu _-)$, where $\mu _-=\frac{\lambda _(\xi _1)}{|\xi _1|}$ is independent of $\xi _1$.

This expression (\ref{energiea1}) only depends on the trace
$ (\tilde{v}^{j}) _{3|y_2=0}(y_1)$ and
 $(\tilde{w}^{k} )_{3|y_2=0}(y_1)$, which
are functions defined on $\Gamma _0$.

\begin{remark}
The important fact in (\ref{energiea1}) is the presence of $|\xi _1|$. This comes from $\int_0^{+\infty} e^{-\lambda _- y_2} \D y_2$ and analogous, on account that this integral is equal to $\frac{C}{| \xi _1| }$.
\end{remark}

We now simplify this last expression using a sesquilinear form
involving pseudo-differential operators.

Then, defining the elliptic pseudo-differential operator $P(y_1,\frac{\partial
}{\partial y_1})$ of order $1/2$ with principal symbol
\begin{equation}\label{pseudodiff}
 (\theta  |\xi _1 |)^{1/2}h(\varepsilon,\xi , y_2),
\end{equation}
and summing over $j$ and $k$, we obtain
\begin{equation}\label{energiea3bis}
a(\tilde{v}^a, \tilde{w}^a)=  \int _{\Gamma _0} P(\frac{\partial
}{\partial s}) ( \tilde{v}_3 )_{|\Gamma
_0}\overline{P(\frac{\partial }{\partial s})
(\tilde{w} _3)_{|\Gamma _0}} \D s.
\end{equation}

\begin{remark}
As we only considered the principal terms for large $|\xi _1 |$,
we may define as well $P(\xi _1)$ by the symbol
\begin{equation}\label{pseudodiff2}
P(\xi _1 )= \theta(1+ |\xi _1 |^2)^{1/4}.
\end{equation}
The corresponding pseudo-differential operator is elliptic of order $1/2$.
\end{remark}
\begin{remark}
We shall use the definition (\ref{pseudodiff2}), which is more pleasant than (\ref{pseudodiff}), as such a $P$ defines
an isomorphism from $H^s ( \Gamma _0 ) $ onto $H^{s+1/2}(\Gamma _0)$, $s \in \R$.
\end{remark}

\subsubsection{The flection sublayer on $\Gamma _0$}\label{6.2.2}
\noindent

The structure of the flection sublayer, see the beginning of Section 6, accounting for the two new boundary conditions $\tilde{w}_3=\partial _n \tilde{w}_3=0$ follows from classical issues in singular perturbation theory, as in \cite{EgorovMeunierSanchez} section 7.1.2, \cite{Vishik} and  \cite{Georgescu}. It is mainly concerned with a drastic change of the normal component $\tilde{w}_3$ (whereas the conditions on $\tilde{w}_1$ and $\tilde{w}_2$ are satisfied). The specific structure is analogous to the layer in \cite{San3}.

The thickness is of order $\delta = \varepsilon ^{1/2}$. This may be easily seen by taking into account only higher order terms in the membrane and the flection systems; eliminating $\tilde{w}_1$ and $\tilde{w}_2$, we obtain an equation for $\tilde{w}_3$. The membrane terms are of order 4 and the flection terms are of order 8. In the layer, the derivatives of order $n$ have an order of magnitude $\mathcal{O}(\frac{\tilde{w}_3}{\delta ^n })$. As both membrane and flection terms are of the same order of magnitude in the layer, we thus have 
$$ \mathcal{O}(\frac{\tilde{w}_3}{\delta ^4 })= \varepsilon ^2\mathcal{O}(\frac{\tilde{w}_3}{\delta ^8 }), $$
which furnishes $\delta = \mathcal{O}(\varepsilon ^{1/2 })$.

It is easily seen (as in \cite{EgorovMeunierSanchez} section 7.1.2) that the presence of this flection sublayer plays a negligible role in the asymptotic behavior. Indeed, proceeding as in the previous membrane layer, we see that the expression analogous to (\ref{energiea3bis}) has the form:
\begin{equation}\label{energieaflexion}
\varepsilon ^2a_0(\tilde{v}^a, \tilde{w}^a)=  \varepsilon ^2\int _{\Gamma _0} P_0(\frac{\partial
}{\partial s}) ( \tilde{v} )_{|\Gamma
_0}\overline{P_0(\frac{\partial }{\partial s})
(\tilde{w} )_{|\Gamma _0}} \D s,
\end{equation}
where $P_0$ is an operator of order $0$. Going on to next Section \ref{dernieresection}, the action of sublayer amounts to change $\mathcal{A}$ to $\mathcal{A}+\varepsilon ^2 \mathcal{C}$ where $\mathcal{C}$ is a smoothing operator. Equivalently, we may change $\mathcal{B}$ to $\mathcal{B}+\mathcal{C}$ which is again a $3$-order operator (as $\mathcal{C}$ is smoothing). The asymptotic behavior does not change. Equivalently, in  (\ref{S.1}), the effect of the sublayer is to change $s $ to $s +\varepsilon ^2 s _0$ where $s _0$ is a smoothing symbol, or $q$ to $q+s_0$ which is again the symbol of an operator of order $2m>0$.

For that reasons, the influence of the sublayer will no more be mentioned.

\subsection{Formulation of the problem in the heuristic asymptotics}
\noindent

Presently, our aim is to formulate problem (\ref{convergenceajout1}) on the space of the $u^a$ with $u \in G^0$. The forms $b(u,v)$ and $\dualprod{f}{v}$ should be written in the framework of our formal asymptotics, for $\tilde{u}^a$ and $\tilde{v}^a$ obtained from $u$ and $v$ defined on $\Gamma _1$ by solving (\ref{pbdirichletajout12}) and modifying $\tilde{u}$ and $\tilde{v}$ with the $\Gamma _0$-layer.

The computation of the $b$-energy form is exactly analogous to that of \cite{EgorovMeunierSanchez} Sec. 5.3. It follows the ideas of the previous section in a much simpler situation. As only the third component is involved in the higher order terms of the form $b$ (see (\ref{defdeB}) and (\ref{gammabis})), we have
\begin{equation}\label{energybending}
b(\tilde{u}^a, \tilde{v}^a) \approx  \int_{\Omega } B^{\alpha \beta \lambda \mu} \partial_{\alpha \beta } \tilde{u}^a_3\partial_{\lambda \mu}\tilde{v}^a_3 \D \xi \D x.
\end{equation}

Moreover, from (\ref{definitiondeg0})--(\ref{pbdirichletajout12}) and according to our approximations analogous to the construction of a parametrix, $\tilde{u}$, $\tilde{v}$ are only significant in a narrow layer adjacent to $\Gamma _1$. The local structure is analogous to (\ref{solgenerale2}) where obviously the decreasing solution inwards the domain should be chosen. This gives the obvious local asymptotics
\begin{equation}\label{sensitive621}
\hat{\tilde{v}}_3(\xi, y)=\hat{v}_3(\xi _1)  e^{\lambda _- ( \xi _1) y_2},
\end{equation}
where $\lambda _-(\xi _1)$ is proportional to $| \xi _1|$. After substitution (\ref{sensitive621}) in (\ref{energybending}) a computation analogous to that of Section \ref{6.2.1} (but much easier) gives (using a partition of unity):
$$b(\tilde{u}^a, \tilde{v}^a)=\Sigma _{j,k} \int_{-\infty}^{+\infty} \zeta _{jk}(y_1)| \xi _1|^3 \tilde{u}^j _3 (\xi_1)\tilde{v}^k _3(\xi_1)\D \xi _1$$
where $\zeta _{jk}(y_1) $ are smooth positive functions on $\Gamma _1$ depending on the coefficients. The function $| \xi _1|^3$ comes obviously from the integrals in the normal direction of products of second order derivatives of functions of the form $e^{\lambda _- ( \xi _1) y_2}$, with $\lambda _-(\xi _1)$ proportional to $| \xi _1|$.

Then, defining the pseudo-differential operator $Q(\frac{\partial
}{\partial y_1})$ of order $3/2$ with principal symbol
\begin{equation}\label{pseudodiffQ}
\sqrt{\zeta(y_1) |\xi _1 |^{3}},
\end{equation}
we have within our approximation:
\begin{equation}\label{energieb3}
\int_{\Omega } B^{\alpha \beta \lambda \mu} \partial_{\alpha \beta } u_3 \partial_{\lambda \mu } v_3  \D x= \int _{\Gamma _1 }Q(\frac{\partial }{\partial y_1})u  \ Q(\frac{\partial }{\partial y_1})v\D y_1.
\end{equation}

We observe that the operator $Q$ is only concerned with the trace
on $\Gamma_{1}$ and $y_1$ which denotes its arc.

The formal asymptotic problem becomes:
\begin{equation}\label{pblaxasymptotic2}
\left\{\begin{array}{l}\textrm{Find }  \tilde{u}^{\varepsilon }\in G \textrm{ such that } \forall \tilde{v} \in G\\
\int _{\Gamma _0} P(\frac{\partial \tilde{u}^{\varepsilon }  }{\partial n})
\overline{P(\frac{\partial \tilde{v}  }{\partial n}) }\D s + \varepsilon ^2
\int _{\Gamma _1} Q (\tilde{u}^{\varepsilon })  \ \overline{Q(\tilde{v})}\D s =
\dualprod{ f}{w},
\end{array}\right.
\end{equation}
where $G$ is the completion of $G^0$ for the norm
$$\| \tilde{v}\|^2_G=\int_{\Gamma _0} \Big | P(\frac{\partial v}{\partial n})\Big| ^2 \D s +\int_{\Gamma _1} \Big | Q(v_3)\Big| ^2 \D s$$

\begin{remark}
For $\varepsilon >0$, (\ref{pblaxasymptotic2}) is a classical Lax-Milgram problem. Continuity and coerciveness follow from the ellipticity of the operators $P$ and $Q$.
\end{remark}

\subsection{The formal asymptotics and its sensitive behaviour}\label{dernieresection}
\noindent

In the sequel, we shall denote
\begin{eqnarray}\label{defdealpha}
\alpha (\tilde{v}^{\varepsilon }, \tilde{w})&=&\int _{\Gamma _0} P(\frac{\partial \tilde{v}^{\varepsilon }  }{\partial n})
\overline{P(\frac{\partial \tilde{w}  }{\partial n})} \D s \\
\beta (\tilde{v}^{\varepsilon }, \tilde{w})&=&\int _{\Gamma _1} Q (\tilde{v}^{\varepsilon })  \ \overline{Q(\tilde{w})} \D s.
\end{eqnarray}

We observe that the problem (\ref{pblaxasymptotic2})
 is again in the same abstract framework
as the initial problem (\ref{pbvar}). Nevertheless, the context is different, as the non-local character of the new problem is apparent from the structure of the space $G$. Let us define the operators
\begin{equation}\label{operateur}
\mathcal{A} \in \mathcal{L}(G, G'), \quad \mathcal{B} \in \mathcal{L}(G, G')
\end{equation}
by
\begin{equation}\label{defdealpha1}
\alpha (v, w)=\dualprod{\mathcal{A} v}{w}  
\quad
\beta (v, w)=\dualprod{\mathcal{B} v}{w} .
\end{equation}

Let $G_{\mathcal{A}}$ be the completion of $G$ with the norm
\begin{equation}
\|v \| _{\mathcal{A}} = \| \mathcal{A}v \|_{G'}.
\end{equation}

Denoting again by $\mathcal{A}$  its extension to $ \mathcal{ L}
(G_{\mathcal{A}}, G')$, which is an isomorphism, we may rewrite 
(\ref{pblaxasymptotic2}) in the form:
\begin{equation}\label{pblaxasymptotic3}
\Big( \mathcal{A} + \varepsilon \mathcal{B}\Big) \tilde{v}^{\varepsilon} =F,
\end{equation}
where $F \in G'$ is defined by
\begin{equation}\label{defdeF}
\dualprod{F}{\tilde{w}}
= \int _{\Omega } f \tilde{w}\D x , \ \forall \tilde{w} \in V.
\end{equation}
It follows that
\begin{equation}\label{convergence}
\tilde{v}^{\varepsilon} \to \tilde{v}^0 \textrm{ strongly in } G_{\mathcal{A}},
\end{equation}
where
\begin{equation}\label{defdev0}
\mathcal{A} \tilde{v}^0 =F.
\end{equation}

\textit{Reduction to a problem on $\Gamma_{1}$}

In order to exhibit more clearly the unusual character of the
problem, we shall now write (\ref{pblaxasymptotic2}) in another,
equivalent form involving only the traces on $\Gamma_{1}$ . Coming back to
(\ref{pbdirichletajout12}), let us define $\mathcal{R}_0$ as follows. For
a given $w \in C^{\infty }(\Gamma _1)$ we solve
(\ref{pbdirichletajout12}) and we obtain
\begin{equation}\label{defdeR}
 \tilde{w}_{3} = \mathcal{R}_0 w.
\end{equation}

Using the regularity properties of the solution of (\ref{pbdirichletajout12}), it follows that $\mathcal{R}_0 w$ is in
$C^{\infty }(\Gamma _0)$. Moreover, we may take in (\ref{pbdirichletajout12}) a $w$ in any $H^s(\Gamma _1)$, $s \in \R$ and
the corresponding solution is of class $C^{\infty }$ on $\Gamma _0$ and its neighbourhood, so that $\mathcal{R}_0$ has an extension
which is continuous from $H^s(\Gamma _1)$ to $C^{\infty }(\Gamma _0)$. We shall denote  by $\mathcal{R}_0$
such an extension, so that
\begin{equation}\label{extensiondeR0}
\mathcal{R}_0 \in \mathcal{L}(H^s(\Gamma _1),H^{r }(\Gamma _0)), \forall s, r \in \R.
\end{equation}
Then, (\ref{pblaxasymptotic2}) may be written as a problem for the traces on $\Gamma _1$:
\begin{equation}\label{pblaxasymptotic3}
\left\{\begin{array}{l}\textrm{Find }  v^{\varepsilon }\in
H^{3/2}(\Gamma _1) \textrm{ such that } \forall w \in
H^{3/2}(\Gamma _1)\\
\int _{\Gamma _0} P(\frac{\partial   }{\partial s})
\mathcal{R}_0v^{\varepsilon } \overline{P(\frac{\partial
}{\partial s}) \mathcal{R}_0w }\D s + \varepsilon ^2 \int _{\Gamma
_1} Q(\frac{\partial   }{\partial s}) v^{\varepsilon }  \
\overline{Q(\frac{\partial }{\partial s})w}\D s = \int_{\Omega }
F\tilde{w} \D x,
\end{array}\right.
\end{equation}
where the configuration space is obviously $H^{3/2}(\Gamma _1)$.
The left hand side with $\varepsilon >0$ is continuous and
coercive.

\begin{remark}
Coerciveness follows from the ellipticity of $Q$, as it is of order $3/2$. Strictly speaking, this only ensures coerciveness on the leading order terms, which may "forget" a finite-dimensional kernel. But this is controlled by $\mathcal{R}_0$, as it is a surjective operator. Indeed, $\mathcal{R}_0 v=0$ implies $\gamma _{\alpha \beta } (\tilde{v} )=0$ with $\tilde{v}_3=\tilde{v}_2=0$ on $\Gamma _0$, which implies $\tilde{v}=0$ (and then $v=0$) using the uniqueness of the Cauchy problem for the rigidity system.
\end{remark}

Here $F \in H^{-3/2}(\Gamma _1)$ is defined for $f \in V'$ by
\begin{equation}\label{defdeF2}
\dualprod{F}{w}_{H^{-3/2}(\Gamma _1), H^{3/2}(\Gamma _1)}=\dualprod{f}{\tilde{w}}.
\end{equation}

We note that, for instance, when the "loading" $f$ is defined by a
"force" $F$ on $\Gamma _1$, this function is the $F$ in
(\ref{pblaxasymptotic3}).
Obviously, (\ref{pblaxasymptotic3}) may be written:
\begin{equation}\label{pbforasymptotic4}
\Big(\mathcal{R}_0^* P^*(\frac{\partial   }{\partial s})
P(\frac{\partial   }{\partial s}) \mathcal{R}_0 + \varepsilon ^{2}
Q^*(\frac{\partial   }{\partial s}) Q(\frac{\partial   }{\partial
s})  \Big) \tilde{v}^{\varepsilon } = F.
\end{equation}
From (\ref{extensiondeR0}) we see that $ \mathcal{R}_0^*  $ is
also a smoothing operator, i. e.:
\begin{equation}\label{defdeRO2}
\mathcal{R}_0^*  \in \mathcal{L}(H^{-r}(\Gamma _1),H^{-s }(\Gamma
_0)), \forall s, r \in \R.
\end{equation}
Now we define the new operators (but we use the same notations)
\begin{eqnarray}\label{operateur2}
\mathcal{A}&=& \mathcal{R}_0^*  P^* P \mathcal{R}_0\  \in
\mathcal{L}(H^s(\Gamma _1),H^{r }(\Gamma _0)), \forall s, r \in \R, \\
\mathcal{B}&=&   Q^* Q\  \in \mathcal{L}(H^{3/2}(\Gamma _1) , H^{-
3/2}(\Gamma _1)).
\end{eqnarray}
Obviously, $\mathcal{B}$  is an elliptic pseudo-differential
operator of order 3, whereas $\mathcal{A}$ is a smoothing (non
local) operator. Then (\ref{pbforasymptotic4}) becomes
\begin{equation}\label{pblaxasymptotic4}
\Big( \mathcal{A} + \varepsilon ^{2}\mathcal{B}\Big)
v^{\varepsilon} =F \textrm{ in } H^{- 3/2}(\Gamma _1).
\end{equation}
Once more, the problem (\ref{pblaxasymptotic3}) is in the general framework of (\ref{pbvar}), so that we can define the space
$\mathcal{V}= H^{3/2}(\Gamma _1)$ and its completion $\mathcal{V}_A$ with the norm
\begin{equation}\label{norme4}
\| v \| _{\mathcal{ A}}=\| \mathcal{ A}v \| _{H^{-3/2}(\Gamma
_1)}.
\end{equation}

Denoting similarly by $\mathcal{A}$ the continuous extension of
$\mathcal{A}$, which is an isomorphism between
$\mathcal{V}_{\mathcal{A}}$ and $\mathcal{V}'$, we obtain
\begin{equation}\label{convergence2}
u^{\varepsilon} \to u^0 \textrm{ strongly in } \mathcal{V}_{\mathcal{A}},
\end{equation}
where $u^0 \in \mathcal{V}_{\mathcal{A}}$ satisfies
\begin{equation}\label{defdeu0}
\mathcal{A} u^0 =F.
\end{equation}

Obviously, this equation is uniquely solvable in $\mathcal{V}_{\mathcal{A}}$ for
$F \in \mathcal{V}'=H^{-3/2}(\Gamma _1)$. But, the unusual character of this equation appears now clearly:
\begin{proposition}
Let $F \in H^{-3/2}(\Gamma _1)$ and $F \notin C^{\infty }(\Gamma
_1)$, then the problem (\ref{defdeu0}) has no $u^0$ solution in
$\mathcal{D}'(\Gamma _1)$.
\end{proposition}
\begin{proof}
If $u^0  \in \mathcal{D}'(\Gamma _0)$ was a solution of
(\ref{defdeu0}), as $\Gamma_{1}$ is compact, $u^0$ should be in
some $H^{s}$, then recalling (\ref{operateur2}), we should have
$\mathcal{A} u^0 \in C^{\infty }(\Gamma _0) $, which is not
possible. Moreover, (\ref{pblaxasymptotic4}) is clearly of the form (\ref{S.1}).
\end{proof}

\noindent

\end{document}